\documentclass{article}
\usepackage[utf8]{inputenc}
\usepackage{turnstile}
\usepackage[leqno]{amsmath}
\usepackage{amssymb,amsthm,lmodern,geometry}
\usepackage[dvipsnames]{xcolor}
\usepackage{enumitem}
\usepackage{centernot}
\usepackage{authblk}
\usepackage{float}
\usepackage{hyperref}
\newtheorem*{thm_soundness}{Theorem~\ref{thm:soundness}}

\DeclareMathSymbol{\mhyph}{\mathalpha}{operators}{`-}

\usepackage{stmaryrd}
\usepackage{comment}
\usepackage{bussproofs}
\usepackage{nicefrac}
\usepackage{proof,color}
\usepackage{multicol}

\newcommand{\leqnomode}{\tagsleft@true}
\newcommand{\reqnomode}{\tagsleft@false}
\makeatother

	\theoremstyle{definition}
		\newtheorem{remark}{Remark}
		\newtheorem{fact}{Fact}
		
	\theoremstyle{plain}
	\newtheorem{dfn}{Definition}
		
			\newtheorem{thm}{Theorem}
		\newtheorem{lemma}{Lemma}
		\newtheorem{prop}{Proposition}
		\newtheorem{corollary}{Corollary}
		
		\newtheorem{obse}{Observation}
		\newtheorem{op}{Open problem}

        \theoremstyle{definition}

        \newcommand{\T}{\mathrm{Tr}}

        \newcommand*\subdot[1]{\oalign{$#1$\cr\hfil.\hfil}}%
        \newcommand{\mc}[1]{\mathcal{#1}}
        \newcommand{\lnat}{\mathcal{L}_\mathbb{N}}
        \newcommand{\lt}{\mc{L}_\T}
        \newcommand{\lpos}{\mathcal{L}^{\mrm{pos}}}
        \newcommand{\mrm}{\mathrm}
        \newcommand{\vphi}{\varphi}
        \newcommand{\corn}[1]{\ulcorner #1\urcorner}
        \newcommand{\beq}{\begin{equation}}
        \newcommand{\eeq}{\end{equation}}

        \newcommand{\ra}{\rightarrow}
        \newcommand{\Ra}{\Rightarrow}
        \newcommand{\pat}{\mrm{PAT}}
        \newcommand{\nat}{\mathbb{N}}
        \newcommand{\sth}{\;|\;}
        \newcommand{\patinf}{\mrm{PA}^\infty_{\mrm{T}}}
        \newcommand{\patpinf}{\mrm{PA(P)}^\infty_{\mrm{T}}}
        \newcommand{\lp}{\mc{L}(P)_\nat}
        \newcommand{\ltp}{\mc{L}(P)_\T}
        \newcommand{\vfmpinf}{\mrm{VFM(P)}^\infty}

\usepackage[numbers]{natbib}
\title{Axiomatic theories of supervaluational truth: \\ completing the picture}
\author[1,2]{Pablo Dopico}
\author[3]{Daichi Hayashi}
\affil[1]{\small Department of Philosophy, King's College London, pablo.dopico@kcl.ac.uk}
\affil[2]{\small Faculty of Philosophy, Humboldt Universität zu Berlin}
\affil[3]{\small Hokkaido University, daichinhayashi0611@gmail.com}
\date{}

\begin{document}

\maketitle

\begin{abstract}
Supervaluational fixed-point theories of formal truth aim to amend an important shortcoming of fixed-point theories based on the Strong Kleene logic, namely, accounting for the truth of classical validities. In a celebrated paper, Andrea Cantini proposed an axiomatization of one such supervaluational theory of truth, which he called VF, and which proved to be incredibly strong proof-theoretically speaking. However, VF only axiomatizes one in a collection of several supervaluational schemes, namely the scheme which requires truth to be consistent. In this paper, we provide axiomatic theories for the remaining supervaluational schemes, labelling these systems VF$^-$ (for the theory which drops the consistency requirement), and VFM (for the theory which requires not only consistency but also completeness, i.e., maximal consistency). We then carry out proof-theoretic analyses of both theories. Our results show that VF$^-$ is as strong as VF, but that VFM's strength decreases significantly, being only as strong as the well-known theory KF. Furthermore, we introduce and analyse proof-theoretically two variants of these theories: the schematic extension, in the sense of Feferman, of VFM; and a theory in-between VFM and VF, that we call VFW, and which drops the assumption of maximal consistency. The former is shown to match the strength of predicative analysis; for the latter, we show its proof-theoretical equivalence with ramified analysis up to the ordinal $\vphi_20$, thus standing halfway between VFM and VF. 
\end{abstract}

\section{Introduction}

Due to the Liar paradox, a common proposal among formal theorists of truth is to take the truth predicate as only partially defined. In \cite{kripke_1975}, Kripke submitted a method to do this based on a fixed-point construction. His idea is as follows: one starts with a set of sentences granted to be true, and then applies an operator that collects the consequences of those sentences on some partial logic. Then, when the operator is monotonic, a fixed-point can be reached.\footnote{There are nuances to this, for the methodology to obtain a fixed-point changes depending on the set of sentences one applies the operator to. While monotonicity is always required, the usual procedure of applying the operator at successor stages and taking unions at limits also demands that the set be sound w.r.t. the operator, i.e., that if $\Phi(\cdot)$ is the operator, then the set $S$ be such that $S\subseteq \Phi(S)$. Fixed-points can also be reached if at limits intersections are taken and the set is closed w.r.t. the operator, i.e., such that $\Phi(S)\subseteq S$---see e.g. \cite{fitting_1986}; or, given any set, by applying more heterodox constructions such as in \cite{cantini_1989}.} The peculiarity (and advantage) of this fixed-points is that, when constructed with the right logic, a sentence $\vphi$ will be in the fixed-point iff the sentence `$\vphi$ is true' is. 

Very often, the partial logic with which the operator is defined is Strong Kleene (SK) logic. Indeed, the Kripke construction over SK logic has been the most explored one so far, and the basis of the celebrated Kripke-Feferman theory. However, Kripke's construction over SK presents what is arguably an important flaw, namely that many classical validities are not necessarily in the fixed-points. For example, `$\lambda\rightarrow\lambda$', for $\lambda$ the Liar sentence, is not in many fixed-points. 

Kripke's suggestion to overcome this drawback is to formulate the construction with some form of supervaluational logic. Since these logics remain partial logics, defining a monotone operator is possible. At the same time, and unlike the SK case, a supervaluational logic allows to recover all classical validities: supervaluational semantics makes true whatever is satisfied by all classical models meeting certain conditions, hence all classical validities are in.

Up to this point, we have been talking about \textit{semantic} theories of truth. But, just like the fixed-point construction over SK has an \textit{axiomatic} counterpart, namely the theory KF, the fixed-point construction over supervaluational logic has its axiomatic version. Or, at least, the fixed-point construction over one such supervaluational logic, for there are three prominent supervaluational satisfaction relations that often get mentioned in the literature. We will explain them in detail below, so for now it will suffice to know the labels they receive: VB, VC and MC.\footnote{There is a fourth relevant supervaluational relation, often labelled SV, which we will also introduce. However, as we shall see, the theory of truth that arises with this scheme has some undesirable features.} Thus, in an influential paper, Andrea Cantini \cite{cantini_1990} proposed an axiomatic theory of truth which he labelled VF and which, to some extent, captures the fixed-point construction over the satisfaction relation VC: the theory is sound with respect to standard models based on these fixed-points. Moreover, Cantini showed that the theory VF over Peano Arithmetic was mathematically very fruitful, in the proof-theoretic sense: it was proof-theoretically equivalent to (meaning it proves the same arithmetical theorems as) the theory of one inductive definition $\mrm{ID}_1$, far beyond Feferman's limits for predicative analysis.  

In this paper, we aim to complete the picture of axiomatic theories of supervaluational truth by providing sound axiomatizations of the remaining satisfaction relations: VB and MC. These are captured by two theories, which we call $\mrm{VF}^-$ and $\mrm{VFM}$. The former is characterized by a lack of axioms concerning the truth predicate in the \textit{internal} theory, whereas the latter includes an axiom of completeness for the truth predicate also in the internal theory. We also offer proof-theoretic analyses of these two theories. Finally, we provide a proof-theoretic analysis of two variants of $\mrm{VFM}$: on the one hand, a theory in-between $\mrm{VF}$ and $\mrm{VFM}$, which drops the internal completeness axiom that characterizes $\mrm{VFM}$ and instead adds the rule of $\T$-Elimination, and which we call $\mrm{VFW}$; on the other hand, the so-called schematic extension---in the sense of \cite{feferman_1991}---of the theory $\mrm{VFM}$, which we call $\mrm{VFM}^*$.

Our results show that:

\begin{itemize}
    \item $\mrm{VF}^-$ is as proof-theoretically strong as $\mrm{VF}$, thus matching the strength of $\mrm{ID}_1$. 
    \item $\mrm{VFM}$ is considerably weaker, and only as strong as KF---matching the strength of ramified analysis up to $\varepsilon_0$.
    \item Dropping internal completeness and adding the rule of $\T$-Elimination, a weakening of the so-called T-Out axiom, strengthens VFW with respect to VFM, but falls short of the proof-theoretic power of VF. 
    \item $\mrm{VFM}^*$ is as proof-theoretically strong as the schematic extension of KF---matching the strength of ramified analysis up to $\Gamma_0$.
\end{itemize}

\section{Preliminaries}

\subsection{Notational preliminaries}

We work with languages whose logical symbols are $\neg, \vee, \wedge, \forall, \exists$, as well as brackets. We write $\varphi \rightarrow \psi$ as an abbreviation for $\neg \varphi \vee \psi$. We follow mostly the conventions in \cite[Ch.5]{halbach_2014}. Our base language $\lnat$ is a definitional extension of the language of Peano Arithmetic ($\mrm{PA}$) with finitely many function symbols for primitive recursive functions. We assume a standard formalization of the syntax of first-order languages, including $\lnat$ itself---see e.g. \cite{hajek_pudlak_2017}. $\lnat$ is assumed to contain a finite set of function symbols that will stand for certain primitive recursive operations. For example, $\subdot \neg$ is a symbol for the primitive recursive function that, when inputted the code of a formula, yields the code of its negation. The same applies to $\subdot \vee, \subdot \forall, \subdot \wedge, \subdot \exists$. 
We assume a function symbol for the substitution function, and write $x (t/v)$ for the result of substituting $v$ with $t$ in $x$; $\ulcorner \varphi(\dot x)\urcorner$ abbreviates $\ulcorner \varphi (v)\urcorner (\mrm{num}(x)/\corn{v}) $, for $x$ a term variable (and provided $\varphi$ only has one free variable). Furthermore, we write $t^\circ$ for the result of applying to a term $t$ the evaluation function (which outputs the value of the inputted term). Note that this is an abbreviation for a formula, and not a symbol of the language.

  Given a language $\mc{L}$, we write  $\text{Var}_{\mc{L}}(v)$ for the formula representing the set of (codes of) variables, $\text{CT}_{\mc{L}}$ for the formula representing the set of (codes of) closed terms of $\mc{L}$, $\mathrm{For}_{\mc{L}}(x)$ ($\mathrm{For}_{\mc{L}}(x,y)$), for the formula representing the set of all formulae (formulae with $v$ free) of $\mc{L}$, and $\mathrm{Sent}_{\mc{L}}(x)$, for the formula representing the set of all sentences of $\mc{L}$. Naturally, we also use $\mrm{Var}_\mc{L}$, $\mrm{CT}_\mc{L}$, $\mrm{For}_\mc{L}$, and $\mrm{Sent}_\mc{L}$ to stand for the sets corresponding to these formulae.  We will occasionally omit reference to $\mc{L}$ when this is clear from the context.  Here, we are mostly interested in the cases in which $\mc{L}$ is $\lnat$ or $\lt:=\lnat\cup \{\T\}$; the exception will be the ramified truth languages, which we will present in time. The expression $\Bar{n}$ stands for the numeral of the number $n$ (although we omit the bar for specific numbers). We write $\# \varphi$ for the code or Gödel number of $\varphi$, and $\ulcorner \varphi\urcorner$ for the numeral of that code. 

$\pat$ is the theory formulated in the language $\lt$ consisting of the axioms of $\mrm{PA}$ with induction extended to the whole language. $\pat$ will be our background syntax theory. 

With the exception of $\mrm{VF}^-$, all the theories studied in this paper will be shown to be \emph{predicative} in strength, that is, their proof-theoretic ordinals are the so-called Feferman--Sch\"{u}tte ordinal $\Gamma_0$ or less. 
For that purpose, we use in particular the techniques of predicative ordinal analysis,
and thus we assume an appropriate ordinal  notation system OT up to $\Gamma_0$, and the ordinal less-than relation $<$ on $\Gamma_0$ (for more detailes, see, e.g., \cite[Ch.3]{pohlers_2009}). For simplicity, we identify each ordinal number with its notation. We use $0$ as the ordinal number, the ordinal sum $\alpha + \beta$, and the Veblen function $\vphi_{\alpha}\beta$. 
In addition, it is convenient to use the \emph{natural sum} $\alpha \# \beta$ (cf.~\cite[p.~32]{pohlers_2009}), which satisfies the following:
\begin{itemize}
    \item $\alpha \# \beta = \beta \# \alpha$,
    \item $\alpha \# \beta < \alpha \# \gamma$ for $\beta < \gamma$.
\end{itemize}
As usual, $1$ is the ordinal number is defined by $\vphi_00$, the $\omega$-exponentiation $\omega^x$ is defined as $\vphi_0x$, and the $\alpha$-th epsilon number $\varepsilon_{\alpha}$ is defined by $\vphi_1\alpha$. 

Given a language $\mathcal{L}$, an ordinal number $\alpha < \Gamma_0$, and an $\mathcal{L}$-formula $A$, transfinite induction for $A$ up to $\alpha$ is defined as the formula: 
\[
\mrm{TI}(\alpha, A):= \forall \beta (\forall \gamma < \beta A(\gamma) \to A(\beta)) \to \forall \beta < \alpha A(\beta).
\]
In line with the latter, the schema $\mrm{TI}_{\mathcal{L}}(<\alpha)$ is defined to be the set $\{ \mrm{TI}(\beta, A) \sth \beta < \alpha \ \& \ A \in \mathcal{L} \}$. Finally, if $S$ is a theory, its proof-theoretic ordinal (denoted $\vert S\vert$) is the ordinal $\alpha$ such that $\mrm{PA}+\mrm{TI}_{\lnat}(<\alpha)$ and $S$ prove the same $\lnat$-statements (and verifiably so within $\mrm{PA}$). If $T$, $S$ are theories, we write $\vert T\vert \geqq \vert S\vert$ to indicate that the proof-theoretic ordinal of $T$ is equal or greater than the proof-theoretic ordinal of $S$. Similarly, $\vert T\vert \equiv \vert S\vert$ is defined as $\vert T\vert \geqq \vert S\vert$ and $\vert S\vert \geqq \vert T\vert$. 

\subsection{Supervaluational theories of truth}\label{section:supervaluational theories}

For a satisfaction relation $\Vdash$, a set of (codes) of sentences $X$, and a formula $A$ of $\lt$, we will write $X \Vdash A$ to abbreviate $(\nat,X, X^-)\Vdash A$, where $X\subseteq \mrm{Sent}_{\lt}$ is the relevant interpretation of the truth predicate and $X^-$ (the \textit{antiextension} of $X$) is defined as $\{\#\vphi \sth \#\neg\vphi\in X\}$. 

As we mentioned, the literature one presents a handful of supervaluational satisfaction relations ($\vDash$ stands for the classical satisfaction relation), namely:
\begin{align}
\label{eq:sva}X\vDash_\mrm{sv} \varphi & \text{ iff }\forall X' (X'\supseteq X \Rightarrow X' \vDash \varphi)\\
\label{eq:svb} X\vDash_\mrm{vb} \varphi &\text{ iff } \forall X' (X'\supseteq X \, \&\,  X' \cap X^- =\varnothing \Rightarrow X' \vDash \varphi)\\
    X\vDash_\mrm{vc} \varphi &\text{ iff } \forall X' (X'\supseteq X\, \&\,  X' \cap X'^-=\varnothing \Rightarrow X' \vDash \varphi)\\
  \label{eq:svd}   X\vDash_\mrm{mc} \varphi &\text{ iff } \forall X' (X'\supseteq X\, \&\, \mrm{MCX}(X') \Rightarrow X' \vDash \varphi)
\end{align}

\noindent In each line of \ref{eq:svb}-\ref{eq:svd}, the second conjunct in the antecedent is called the \textit{admissibility condition}. In \ref{eq:svd}, $\mrm{MCX}$ stands for the set of all maximally consistent sets of codes of sentences. 

With this, one obtains fixed-points for the interpretation of the truth predicate in the following way. First, given a set of codes of sentences $S$ and some satisfaction relation $e\in \{\mrm{sv}, \mrm{vb}, \mrm{vc}, \mrm{mc}\}$, one defines the \emph{Kripke jump} as follows: 

\beq
\mathcal{J}_\mrm{e}(S):=\{\#\varphi\sth S \vDash_e \varphi\}
\eeq

The fixed-point interpretation of $\T$ is now obtained by the following transfinite construction, starting with some extension $S$:

\begin{align*}
& \Gamma_0 =  S\\
		& \Gamma_{\alpha+1}= \mathcal{J}_\mrm{e}(\Gamma_\alpha)\\
		& \Gamma_\lambda = \bigcup_{\beta<\lambda} \Gamma_\beta\text{ for }\lambda \text{ a limit ordinal}\\
\end{align*}

One can easily verify that the jump operator $\mathcal{J}_\mrm{e}$, for $e\in \{\mrm{sv}, \mrm{vb}, \mrm{vc}, \mrm{mc}\}$, is monotonic, that is, 

\beq
    E\subseteq E' \Ra \mathcal{J}_\mrm{e}(E)\subseteq \mathcal{J}_\mrm{e}(E')
\eeq

As we hinted before, and by the theory of inductive definitions (see e.g. \cite[Chs. 4 and 5]{mcgee_1991}), this entails that the operator will have fixed-points, i.e., there will be sets $S$ such that $\mc{J}_e(S)=S$. Moreover, there will also be a least or minimal fixed-point $\mrm{I}_e$, in the sense that, for any other fixed-point $\mrm{F}_e$, $\mrm{I}_e\subseteq \mrm{F}_e$.

In connection with the above, in \cite{cantini_1990} Cantini formulated the axiomatic theory VF, defined as follows:

\begin{dfn}
[$\mrm{VF}$]
The $\lt$-theory $\mrm{VF}$ consists of $\mrm{PAT}$ with the universal closure of the following axioms:

\begin{description}
\item[$(\mrm{V}1)$] $\mathrm{CT}(x)\land \mathrm{CT}(y) \to [(\T(x\subdot{=}y) \leftrightarrow x^{\circ} = y^{\circ})\wedge (\T(x\subdot{\neq}y) \leftrightarrow x^{\circ} \neq y^{\circ})]$

\item[$(\mrm{V}2)$] $\mathrm{Ax}_{\mrm{PAT}}(x) \to \T(x)$

\item[$(\mrm{V}3)$] $\forall z \T(x ( z/v)) \to \T(\subdot{\forall}vx)$

\item[$(\mrm{V}4)$] $\T(x) \rightarrow \T\ulcorner \T(\dot{x}) \urcorner $

\item[$(\mrm{V}5)$] $\mrm{Sent}_{\lt}(x)\wedge \T\corn{\neg \T (\dot x)} \ra \T (\subdot \neg x)$

\item[$(\mrm{V}6)$] $\T(x\subdot{\to}y) \to (\T(x) \to \T(y))$

\item[$(\mrm{V}7)$] $\T\ulcorner \neg (\T(\dot x)\wedge \T(\subdot \neg\dot  x))\urcorner $

\item[$(\mrm{V}8)$] $\T(\ulcorner \T(\dot{x}) \to \mathrm{Sent}_{\lt}(\dot{x})\urcorner)$

\item[$(\mrm{V}9)$] $\T\ulcorner \vphi\urcorner \rightarrow \vphi$ for any formula $\vphi$ of $\lt$

\end{description} 
\end{dfn}

The interesting relation with Kripke's theory of truth, which Cantini devised, is that the models of this theory were provided by fixed-point models of the scheme VC.\footnote{But not only. It is known that the set of stable truths following Herzberger's revision theory is also a model of $\mrm{VF}$ (see \cite[\S 69]{cantini_1996}). } A fixed-point model of a scheme $e$ is a model $(\nat, X)$ for the language $\lt$, where $X$ is the interpretation of $\T$ and $X=\mc{J}_{\mrm{e}}(X)$. Thus, he showed:

\begin{prop}[Cantini]
$\mrm{VF}$ is sound with respect to $\mrm{vc}$ fixed-point models, i.e., if $X=\mc{J}_\mrm{vc}(X)$, then $(\nat, X)\vDash \mrm{VF}$.
\end{prop}

Moreover, he carried out the ordinal analysis of VF: 

\begin{prop}[Cantini]
$\vert \mrm{VF}\vert \equiv \vert \mrm{ID}_1\vert$
\end{prop}
    
This makes VF a remarkably strong theory of truth, one of the strongest in the literature and by far the strongest theory that can be considered an axiomatization of any of Kripke's semantic constructions.

In addition to VF, Cantini also produced an axiomatization of SV, which he calls $\mrm{VF}_p$, and which was shown to be as proof-theoretically strong as $\mrm{ID}_1$---see \cite[Ch. 12] {cantini_1996}. The details are no different from the case of VF, and in both cases it consists in showing how to interpret the theory $\mrm{ID}_1^{acc}$ of accesibility inductive definitions. It is worth-noting, however, that the scheme SV leads to some undesirable consequences. In particular, fixed-points of the scheme SV fail to call untrue any sentence whatsoever, even clear arithmetical falsities; thus, for example, $\neg \T\corn{0=1}$ is not true in this scheme. Accordingly, the axiomatization of this theory cannot declare true that kind of sentences, i.e., $\T\corn{\neg \T \corn{0=1}}$ is not derivable in the theory. 

\section{The theory $\mrm{VF}^-$}

As stated, our goal is to axiomatize the satisfaction relations VB and MC, and thus we start with the former, whose proof-theoretic analysis also turns out to be the easiest. We recall that the satisfaction relation VB is defined as: 

\begin{equation*}
    X\vDash_\mrm{vb} \varphi \text{ iff } \forall X' (X'\supseteq X \, \&\,  X' \cap X^- =\varnothing \Rightarrow X' \vDash \varphi)
\end{equation*}

As one can read off from this definition, given a set $X$, the VB relation requires that the extensions of $X$ to be considered for the interpretation of the truth predicate be consistent with $X$. Hence, this is a form of conditional consistency: if $\vphi\in X$, and $X'$ is one of such extensions, $X'\nvDash \T\corn{\neg \vphi}$ must hold, and so $\neg\T\corn{\neg \vphi}$ must be in $\mrm{VB}(X)$. On the other hand, in the case of VF the consistency requirement is unconditional. Therefore, the latter case is captured by an \textit{internal} consistency axiom (V7), whereas in the case of the axiomatization of the scheme VB it will be captured by a conditional---see axiom $\mrm{VF}^-7$ below. Since this is the only major change, we now provide such axiomatization, which we call $\mrm{VF}^-$:\footnote{There is a further difference between VF and $\mrm{VF}^-$: $\mrm{VF}^-5$ has now been turned into a biconditional. The reason is that the right-to-left direction of this axiom is provable in VF via axiom V7; and $\mrm{VF}^-7$, which is a weakening of V7, does not suffice to prove the claim. Nonetheless, all fixed-point models of VB satisfy $\mrm{VF}^-5$ as displayed.}

\begin{dfn}
[$\mrm{VF}^-$]\label{vfminus}
The $\lt$-theory $\mrm{VF}^-$ consists of $\mrm{PAT}$ with the universal closure of the following axioms:

\begin{description}
\item[$(\mrm{VF}^-1)$] $\mathrm{CT}(x)\land \mathrm{CT}(y) \to [(\T(x\subdot{=}y) \leftrightarrow x^{\circ} = y^{\circ})\wedge (\T(x\subdot{\neq}y) \leftrightarrow x^{\circ} \neq y^{\circ})]$

\item[$(\mrm{VF}^-2)$] $\mathrm{Ax}_{\mrm{PAT}}(x) \to \T(x)$

\item[$(\mrm{VF}^-3)$] $\forall z \T(x( z/v)) \to \T(\subdot{\forall}vx)$

\item[$(\mrm{VF}^-4)$] $\T(x) \rightarrow \T\ulcorner \T(\dot{x}) \urcorner $

\item[$(\mrm{VF}^-5)$] $\mrm{Sent}_{\lt}(x)\ra( \T\corn{\neg \T (\dot x)} \leftrightarrow \T (\subdot \neg x))$

\item[$(\mrm{VF}^-6)$] $\T(x\subdot{\to}y) \to (\T(x) \to \T(y))$

\item[$(\mrm{VF}^-7)$] $\T x\ra \T \corn{\neg \T (\subdot \neg \dot x)}$

\item[$(\mrm{VF}^-8)$] $\T(\ulcorner \T(\dot{x}) \to \mathrm{Sent}_{\lt}(\dot{x})\urcorner)$

\item[$(\mrm{VF}^-9)$] $\T\ulcorner \vphi\urcorner \rightarrow \vphi$ for any formula $\vphi$ of $\lt$

\end{description} 
\end{dfn}

\begin{prop}
$\mrm{VF}^-$ is sound w.r.t. the scheme $\mrm{VB}$. That is: if $X$ is consistent and $X=\mc{J}_{vb}(X)$, then $(\nat, X)\vDash \mrm{VF}^-$.
\end{prop}

\begin{proof}
For axioms $\mrm{VF}^-1$-$\mrm{VF}^-7$ and $\mrm{VF}^-9$, the proof is essentially the same as for axioms V1-V7 and V9 from \cite[Prop. 3.4]{cantini_1990}. The only addenda is the right-to-left direction of $\mrm{VF}^-5$: assuming $S=\mrm{VB}(S)$ and $\mrm{Sent}(x)$, we let $\vphi$ be such that $\corn{\vphi}=x$. Then assume $S\vDash \T\corn{\neg\vphi}$. This means $\neg \vphi\in S$, so for any $S'\supseteq S$ s.t. $S'\cap S^-=\varnothing$, $\vphi\notin S'$, hence $S'\vDash \neg \T\corn{\vphi}$. Therefore, $\neg \T\corn{\vphi}\in \mrm{VB}(S)=S$, and this yields the claim.

As for $\mrm{VF}^-8$, the proof has been informally sketched right before Definition \ref{vfminus}.
\end{proof}

\begin{prop}
$| \mrm{VF}^-| \equiv | \mrm{VF}| \equiv |\mrm{ID}_1|$.
\end{prop}
\begin{proof}
For the lower-bound: By a result of Friedman and Sheard \cite[p. 19] {friedman_sheard_1987}, axioms $\mrm{VF}^-2, \mrm{VF}^-3$, $\mrm{VF}^-6$ and $\mrm{VF}^-9$ are enough to build a model of (parameterless) Bar Induction ($\mrm{BI}'$), which is proof-theoretically equivalent to $\mrm{ID}_1$. Note that this proof is different from the one Cantini gave for VF and which we mentioned in Section \ref{section:supervaluational theories}.

For the upper-bound: $\mrm{VF}^-$ is a subtheory of $\mrm{VF}$---$\mrm{VF}^-7$ can easily be proven in VF from V4 and V7, and the right-to-left direction of $\mrm{VF}^-5$ is given in \cite[Prop. 2.1]{cantini_1990}.
\end{proof}

As such, the theory $\mrm{VF}^-$ is not very interesting from the proof-theoretic point of view; and, from the truth-theoretic perspective, its only appeal is as a sound axiomatization of the scheme VB. On the other hand, the axiomatization of the scheme MC, to which we turn now, is much more interesting proof-theoretically.

\section{The theories VFM$^-$ and VFM}

\subsection{Presentation}

As before, we recall the satisfaction relation MC:

\begin{equation*}
    X\vDash_\mrm{mc} \varphi \text{ iff } \forall X' (X'\supseteq X\, \&\, \mrm{MCX}(X') \Rightarrow X' \vDash \varphi)
\end{equation*}

We can see how this relation includes a requirement, or admissibility condition, on the extensions of the interpretation of the truth predicate. According to this requirement, such an extension must be a maximally consistent set of formulae. The feature of the internal theory that this requirement forces is captured by axiom VF7 below: an axiom of internal completeness and consistency for the truth predicate. 

\begin{dfn}
    $\mrm{VFM}$ is the theory formulated in the language $\mc{L}_\T$ and consisting of the axioms of $\mrm{PAT}$ plus the universal closure of the following axioms:

    \begin{description}
    \item[($\mrm{VF}1)$] $\mrm{CT}(x)\wedge \mrm{CT}(y)\rightarrow [(\T(x\subdot{=}y)\leftrightarrow x^\circ = y^\circ )\wedge (\T(x\subdot{\neq} y)\leftrightarrow x^\circ\neq y^\circ)]$
    \item[($\mrm{VF}2)$] $\mrm{Ax}_{\mrm{PAT}}(x)\rightarrow \T(x)$
    \item[($\mrm{VF}3)$] $\forall z\T(x\, ( z/v))\rightarrow\T(\subdot{\forall} v x)$
    \item[($\mrm{VF}4)$]$\T(x)\leftrightarrow \T(\ulcorner \T(\dot x)\urcorner)$
    \item[($\mrm{VF}5)$]$\mrm{Sent}(x)_{\mc{L}_\T}\ra (\T(\subdot \neg\ulcorner \T(\dot x)\urcorner ))\leftrightarrow \T(\subdot{\neg} x))$
    \item[($\mrm{VF}6)$] $\T(x\subdot{\rightarrow} y)\rightarrow (\T(x)\rightarrow \T(y))$
    \item[($\mrm{VF}7)$] $\T(\ulcorner \neg \T(\dot x) \leftrightarrow  \T(\subdot \neg \dot x)\urcorner)$
    \item[($\mrm{VF}8)$] $\T(\ulcorner \T(\dot x) \rightarrow\mrm{Sent}_{\mc{L}_\T}(\dot x)\urcorner)$ 
    \end{description}

\end{dfn}

\begin{remark}
Axiom VF5 is redundant in the presence of VF4 and VF7. 
\end{remark}

We now present a weakening---in fact a subtheory---of VFM which allows for a more general batch of proof-theoretic results. The new theory drops the so-called T-Del axiom. 

\begin{dfn}
$\mrm{VFM}^-$ is the theory formulated in the language $\mc{L}_\T$ and consisting of the axioms of $\mrm{VFM}$ VF1-VF3 and VF5-VF8 plus the axiom: 

\begin{itemize}
\item[$\mrm{VF4}^*$]$\T(x)\rightarrow \T(\ulcorner \T(\dot x)\urcorner)$
\end{itemize}

\end{dfn}

\begin{prop}\label{prop:soundness VFM}
$\mrm{VFM}$ is sound w.r.t. the scheme $\mrm{MC}$. That is, if $X$ is consistent and $X=\mc{J}_\mrm{mc}(X)$, then $(\nat, X)\vDash \mrm{VFM}$.
\end{prop}

\begin{proof}
The only extra work with respect to VF and $\mrm{VF}^-$ is the right-to-left direction of VF4 and the left-to-right direction of VF5, since VFM does not have the schema T-Out anymore. 

For the former: we need to show that $\#\T\corn{\vphi} \in X\Ra \#\vphi\in X $. Assume $\#\T\corn{\vphi} \in X$, then $\forall X'\supseteq X(\mrm{MCX}(X') \Ra X'\vDash \T\corn{\vphi})$. So $\forall X'\supseteq X(\mrm{MCX}(X') \Ra \#\T\corn{\vphi}\in X')$. This entails that $X\cup \{\neg \vphi\}$ is inconsistent. Now, there is a set of codes of sentences $Y$ such that $X=\{\#\psi \sth \forall Y'\supseteq Y(\mrm{MCX}(Y')\Ra Y'\vDash \psi\}$---clearly $X$ itself can be that $Y$, since $X$ is a fixed-point. But since $X\cup\{\neg \vphi\}$ is inconsistent, we have $\forall Y'\supseteq Y(\mrm{MCX}(Y')\Ra Y'\vDash \vphi)$. So $\#\vphi \in  X$. 

The case of VF5 is similar. 
\end{proof}

\begin{corollary}
    $\mrm{VFM}$ and $\mrm{VFM}^-$ are consistent. 
\end{corollary}

\subsection{Proof-theoretic analysis of VFM: lower-bound}\label{subsec:VFM-lower-bound}

We will analyse the proof-theoretic strength of both $\mrm{VFM}^-$ and $\mrm{VFM}$. Since $\mrm{VFM}^-$ is a subtheory of $\mrm{VFM}$, we will provide the proof-theoretic lower bound of $\mrm{VFM}^-$, and the proof-theoretic upper bound of $\mrm{VFM}$. In seeing that they coincide, we will establish: 

\begin{equation}
    \vert \mrm{VFM} \vert \equiv \vert \mrm{VFM}^- \vert \equiv \vert \mrm{KF}\vert \equiv\vert \mrm{PUTB}\vert\equiv\vert \mrm{RT}_{<\varepsilon_0}\vert\equiv\vert \mrm{RA}_{<\varepsilon_0}\vert
\end{equation}

In this section, we prove:

\begin{thm}\label{thm:lower_bound_VFM^-}
$\vert \mrm{VFM}^- \vert \geqq \vert \mrm{KF}\vert\equiv\vert \mrm{RT}_{<\varepsilon_0}\vert$
\end{thm}

Since we will be interpreting $\mrm{RT}$-like theories, let us first define the language in which they are formulated:

\begin{dfn}
For an ordinal $\gamma\leq\Gamma_0$, $\mc{L}_{<\gamma}$ is the language extending $\lnat$ with all truth predicates $\T_\beta$, for $\beta<\gamma$, if $\gamma>0$; and the language $\lnat$ if $\gamma=0$. $\mc{L}_{\gamma}$ is defined as $\mc{L}_{\gamma}=\mc{L}_{<\gamma+1}$.
\end{dfn}

As a notational abuse, we write $\mrm{For}_{<\alpha}$, $ \mrm{For}_{\alpha}$, $ \mrm{Sent}_{<\alpha}$, $ \mrm{Sent}_{\alpha}$, etc. in place of $\mrm{For}_{\mc{L}_{<\alpha}}$, $\mrm{For}_{\mc{L}_{\alpha}}$, $\mrm{Sent}_{\mc{L}_{<\alpha}}$, $ \mrm{Sent}_{\mc{L}_{\alpha}}$, and so on. Then, for any $\gamma\leq\Gamma_0$, $\mrm{RT}_{<\alpha}$ is the theory in the language  $\mc{L}_{<\alpha}$ given by PA plus the axioms introduced in e.g. \cite[Def. 9.2]{halbach_2014}. The purpose of this subsection is thus to show how to define the truth predicates of $\mrm{RT}_{<\varepsilon_0}$ in $\mrm{VFM}^-$. The techniques draw on the proof of \cite[Lemma 36]{fujimoto_2010}. We start with a definition: 

\begin{dfn}

Let $\xi(x,X)$ be the following formula: 

\begin{align*}
    &  x\in \mrm{True}_0\\
        &\vee \exists y (x=(\subdot\neg\subdot\neg y) \wedge y \in X)\\
		&\vee  \exists y,z(x=(y\subdot{\vee} z) \wedge (y\in X\vee z\in X))\\
        &\vee  \exists y,z(x=\subdot\neg(y\subdot{\vee} z) \wedge (\subdot\neg y\in X\wedge \neg\subdot z\in X))\\
		&\vee  \exists y,z(x=(y\subdot{\wedge} z) \wedge (y\in X\wedge z\in X))\\
        &\vee  \exists y,z(x=\subdot\neg(y\subdot{\wedge} z) \wedge (\subdot\neg y\in X\vee \neg\subdot z\in X))\\
		&\vee  \exists y,v(x=(\subdot{\forall} vy) \wedge \forall z(y (z/v)\in X))\\
        &\vee  \exists y,v(x=\subdot \neg (\subdot{\forall} vy) \wedge \exists z(\subdot \neg y (z/v)\in X))\\
		&\vee  \exists y,v(x=(\subdot{\exists} vy) \wedge \exists z(y (z/v)\in X))\\
        &\vee  \exists y,v(x=\subdot \neg (\subdot{\exists} vy) \wedge \forall z(\subdot \neg y (z/v)\in X))\\
		&\vee  \exists t(x=(\subdot \T (t)) \wedge t^\circ \in X)\\
		&\vee  \exists t(x=(\subdot \neg \subdot \T (t) )\wedge (\subdot{\neg}t^\circ)\in X\vee \neg \mrm{Sent}(t^\circ))
\end{align*}

\noindent Here, $x\in \mrm{True}_0$ denotes that $x$ is a true arithmetical atomic or negated atomic formula; and $x\in X$ is an abuse of notation, standing for $X(x)$. 

\end{dfn}

It is easy to see that $\xi(x,\T)$ is the arithmetical formula that defines a Kripke set using Strong Kleene logic. The following is also clear, given the axioms of $\mrm{VFM}^-$:

\begin{prop}\label{xi implies T}
$\mrm{VFM}^-\vdash \xi(x, \T) \rightarrow \T x$
\end{prop}

Then, we can define:

\begin{dfn}
$\xi^*(x):=\xi(x,\T) \vee \xi(\subdot\neg x, \T)$.
\end{dfn}

We prove some properties of this formula over $\mrm{VFM}^-$:

\begin{lemma}\label{properties xi star}
$\mrm{VFM}^-$ proves the following: 

\begin{itemize}
    \item[i)] $\xi^*(x)\rightarrow \xi^*(\subdot\neg x)$
    \item[ii)] $\xi^*(x)\wedge \xi^*(y)\rightarrow \xi^*(x\subdot \vee y) \wedge \xi^*(x\subdot \rightarrow y)$
    \item[iii)] $\mrm{Sent}_{\mc{L}_\T}(\subdot \forall vy) \wedge \forall t (\xi^*( y (t/v)))\rightarrow \xi^*(\subdot \forall vy)$
    \item[iv)] $\xi^*(x)\rightarrow (\T \subdot \neg x \leftrightarrow \neg \T x)$
    \item[v)] $\xi^*(x\subdot\vee y)\rightarrow (\T x\subdot \vee y \leftrightarrow \T x \vee \T y)$ 
     \item[vi)] $\xi^*(x\subdot\rightarrow y)\rightarrow (\T x\subdot \rightarrow y \leftrightarrow \T x \rightarrow \T y)$ 
    \item[vii)] $\xi^*(x)\rightarrow (\T \subdot \forall v x \leftrightarrow \forall y \T x (y/v))$
    \item[viii)] $\xi^*(x)\rightarrow (\T x\leftrightarrow \T (\subdot \T x))$
\end{itemize}
\end{lemma}

\begin{proof}
i) - iii) are easy to check. vii) follows directly from VF3 and the closure of the truth predicate under logic. 

For iv): note that $\mrm{VFM}^-$ proves T-Cons (i.e. $\neg (\T x \wedge \T \subdot \neg x)$ for all $x$). Accordingly, the left to right direction of the biconditional in the consequent always holds. Now, we reason informally, assuming $\xi^*(x)$ and $\neg \T x$. If $\xi(x, \T)$, then by Proposition 1, we get $\T x \wedge \neg \T x$, which is a contradiction. So $ \xi(\subdot\neg x, \T)$ must hold; but then, by the same proposition, $\T\subdot \neg x$, as desired. 

vi) can be shown in a similar way. 

For v): as before, the right to left direction of the biconditional in the consequent is given by the axioms of $\mrm{VFM}^-$. For the left to right, we reason informally once more. We assume $\xi^* (x\subdot \vee y)$ and $\T (x\subdot \vee y)$. If $\xi(x\subdot \vee y, \T)$, then by the definition of $\xi(x, \T)$, it follows that $\T x \vee \T y$. If $\xi(\subdot \neg (x\subdot \vee y), \T)$, Proposition 1 gives us $\T \subdot \neg (x\subdot \vee y)$. But, together with the assumption $\T (x\subdot \vee y)$, a bit of logic inside $\T$, and the provability of T-Cons, we get a contradiction, hence anything follows (including $\T x \vee \T y$).

For viii): the left-to-right direction is clear from VF4*. For the right-to-left: if $\xi(x, \T)$, then $\T x$ by Proposition \ref{xi implies T}, so also $\T\subdot \T x\ra \T x $. If $\xi(\subdot \neg x, \T)$, then $\T \subdot \neg x$, so $\xi (\subdot \neg \subdot \T x, \T)$, hence $\T (\subdot \neg\subdot \T x)$. But the assumption that $\T \subdot \T x$, together with the provability of T-Cons, implies anything, thus also $\T x$.
\end{proof}

Now, we define the two primitive recursive functions employed by Fujimoto. For this, we assume some standard coding of ordinals in theories of arithmetic: 

\begin{dfn}
For $\beta< \varepsilon_0$, let $h(x, \beta) $ be a primitive recursive function from such that 

\[   
h(x,\beta) = 
     \begin{cases}
       x &\quad\text{if } x\in \mrm{Sent}_{<\beta}\\
       \ulcorner 0=1\urcorner &\quad\text{otherwise} \\
     \end{cases}
\]

We write $h_\beta(x)$ for $h(x, \beta)$.

\end{dfn}

\begin{dfn}
    With $h_\beta(x)$ defined as above, let $k(x)$ be a primitive recursive function such that

    \[   
k(x) = 
     \begin{cases}
       x, &\quad\text{if } x\in \mrm{At}_{\lnat}\\
       \ulcorner \T (k\circ h_\beta(t))\urcorner &\quad\text{if } x=\subdot \T_\beta t \text{ and } t \text{ is a closed term}\\
       \ulcorner \neg \T k(\dot y)\urcorner  &\quad\text{if } x=\subdot \neg y\\
       \ulcorner \T k(\dot y) \vee \T k(\dot z)\urcorner  &\quad\text{if } x=y \subdot \vee z\\
       \ulcorner \T k(\dot y) \wedge \T k(\dot z)\urcorner  &\quad\text{if } x=y \subdot \wedge z\\
       \ulcorner \forall z \T k(\dot y (u/\corn{z}))\urcorner(z/\corn{u})  &\quad\text{if } x=\subdot \forall z y \text{ and } z \text{ is a variable}\\
       \ulcorner 0=1\urcorner &\quad\mrm{otherwise} \\ 
     \end{cases}
\]
\end{dfn}

\noindent The convulated expression $\ulcorner \forall z \T k(\dot y (u/\corn{z}))\urcorner(z/\corn{u})$ informally stands for: $k(\subdot \forall z\vphi(z))=\forall z \T (k\corn{\vphi (\dot z)}$. This is enough for the purposes of proving our theorem: 

\begin{proof}[Proof of Theorem 1]

Given the properties proved in Lemma \ref{properties xi star}, the main result we want to obtain is the following: for each $\beta<\varepsilon_0$,

\begin{equation}\label{eqn1}
    \mrm{VFM}^-\vdash \forall \gamma\leq \beta (\mrm{Sent}_{<\gamma}(x) \rightarrow \xi^*(kx))
\end{equation}

Since transfinite induction up to $\varepsilon_0$ is provable in PA, we need just prove the progresiveness of the formula in question. So assume the claim holds up to $\delta$; we show it holds for $\delta$. Let $x\in \mrm{Sent}_{<\delta}$. Then, we proceed by induction on the complexity of the formula coded by $x$. 

\begin{itemize}
    \item If $x$ is an atomic formula of arithmetic, then $\xi(x, \T)$ or $\xi(\subdot\neg x, \T)$ is clear, whence $\xi^*(kx)$ holds. 
    \item If $x$ is of the form $\T_\zeta t$ for some $\zeta<\delta$: First, note that

    \begin{align}
        [\T(k \circ h_\zeta t^\circ) \vee \T(\subdot \neg k \circ h_\zeta t^\circ)] & \ra [\xi(\ulcorner \T (k \circ h_\zeta \dot t)\urcorner, \T) \vee \xi(\subdot \neg \ulcorner \T (k \circ h_\zeta \dot t)\urcorner,\T)]\label{line 1}\\
        & \ra \xi^*(kx) \label{line 2}
    \end{align}

The first implication is by the definition of $\xi(x,\T)$, and the second one is the definition of $\xi^*(kx)$.

Now, if $t^\circ\in \mrm{Sent}_{<\zeta}$, then $h_\zeta(t^\circ)=t^\circ$, whence $k\circ h_\zeta(t^\circ)=k(t^\circ)$. By induction hypothesis, we have $\xi^*(k(t^\circ))=\xi^*(k\circ h_\zeta(t^\circ))$, whence the antecedent in (\ref{line 1}) follows. Hence, $\xi^*(kx)$. If $t^\circ\notin \mrm{Sent}_{<\zeta}$, then $k\circ h_\zeta(t^\circ)=\ulcorner 0=1\urcorner$, and we have $\xi(\subdot \neg \ulcorner 0=1\urcorner, \T)$, from where the claim follows. This completes the base case of the induction.

\item If $x=\subdot \neg y$ or $x= y \subdot \vee z$, then we will use the following result, provable in $\mrm{VFM}^-$ via VF4*, for $x\in \mrm{Sent}_{<\delta}$: $\xi^*(kx) \rightarrow \xi^*(\ulcorner \T kx\urcorner)$. The claim can be then proved by using the properties i) and ii) in Lemma \ref{properties xi star}. For instance, for $x=\subdot \neg y$: 

\begin{align*}
    \mrm{Sent}_{<\delta}(\subdot \neg y) & \rightarrow \mrm{Sent}_{<\delta}(y)\\
    &\rightarrow \xi^*(ky)\\
    &\rightarrow \xi^*(\ulcorner \T k(\dot y)\urcorner )\\
    & \rightarrow \xi^*(\ulcorner \neg \T k(\dot y)\urcorner)\\
    & \rightarrow \xi^*(k \subdot\neg y )
\end{align*}

The second line is by IH; the fourth line is by Lemma \ref{properties xi star}, i). 

For the case $x=\subdot \forall v y$, note that the above extends to claim that the following is provable in $\mrm{VFM}^-$, for $x\in \mrm{Sent}_{<\delta}$: $\xi^*(k x(y/v)) \rightarrow \xi^*(\ulcorner \T k(\dot x (y/\ulcorner v\urcorner))\urcorner(y/\ulcorner v\urcorner))$. Then, a very similar argument yields the desired claim, using Lemma \ref{properties xi star}, iii). This completes our induction

\end{itemize} 
Finally, for any ordinal $\alpha<\varepsilon_0$, we define a translation function $\sigma_\alpha(x): \mc{L}_{<\alpha}\rightarrow \mc{L}_\T$ which leaves arithmetical formulae intact, commutes with connectives and quantifiers, and translates $\T_\beta(x)$, for $\beta<\alpha$, as $\T (kx)$. Result (\ref{eqn1}) proved above, together with properties iv-viii of Lemma \ref{properties xi star}, and the axiom VF1, allow to prove the axioms of $\mrm{RT}_{<\varepsilon_0}$ in a straightforward way. For example, the axiom for disjunction of $\mrm{RT}_{<\varepsilon_0}$: 

\begin{center}
    Sent$_{\mathcal{L}_\beta}(x\subdot \vee y)\rightarrow (\T_\beta(x\subdot \vee y)\leftrightarrow \T_\beta x\vee \T_\beta y)$, for any $\beta<\varepsilon_0$
\end{center}

We reason infomally in $\mrm{VFM}^-$. Assume $\mrm{Sent}_{<\beta}(x\subdot \vee y)$. Then, $\xi^*(k(x\subdot \vee y))$, by (\ref{eqn1}). So $\xi^*(\ulcorner \T k(\dot x) \vee \T k(\dot y)\urcorner)$. Then, by item v) in Lemma \ref{properties xi star}, $\T\ulcorner \T k(\dot x) \vee \T k(\dot y)\urcorner\leftrightarrow (\T \ulcorner \T k(\dot x)\urcorner \vee \T \ulcorner \T k(\dot y)\urcorner) $. But note that the left-hand side of the biconditional is just $\sigma_\beta(\T_\beta (x\subdot \vee y))$, and the right-hand side is just $\sigma_\beta( \T_\beta x\vee  \T_\beta y)$ (on the assumption that $\mrm{Sent}_{<\beta}(x\subdot \vee y)$).

\end{proof}

\begin{remark}
Note that the above is a small refinement over Fujimoto's result. Hence, we do not need a formula $D^+(x)$ that holds iff $\mrm{Sent}_{\mc{L}_\T}(x)\wedge (\T x \vee \T\subdot \neg x)\wedge \neg (\T x \wedge \T \subdot \neg x)$. Rather, we need a formula $D(x)$ that entails (but need not be entailed by) $\mrm{Sent}_{\mc{L}_\T}(x)\wedge (\T x \vee \T\subdot \neg x)\wedge \neg (\T x \wedge \T \subdot \neg x)$. Then, as long as it meets conditions i-viii) and one can show both VF1 and $D(x)\rightarrow D(\ulcorner \T k\dot x\urcorner)$ (for the relevant formulae coded by $x$), the result can be proven. 
\end{remark}

This completes the lower-bound proof of $\mrm{VFM}^-$. Since the proof of the upper bound of $\mrm{VFM}$ is more complex and winding, we will dedicate a whole section to it.

\section{Proof-theoretic analysis of VFM: upper-bound}\label{section:upper_bound_vfm}

This section is devoted to showing the upper-bound of VFM, which is evidently an upper-bound for $\mrm{VFM}^-$ too. In particular, we prove:

\begin{thm}\label{upper bound}
$\vert \mrm{VFM} \vert \leqq \vert \mrm{KF}\vert\equiv\vert \mrm{RT}_{<\varepsilon_0}\vert$
\end{thm}

The upper-bound will be given by a cut-elimination argument; the applicability of these arguments for determining the proof-theoretic ordinal of truth theories started with Cantini's upper-bound proof of $\mrm{KF}$ \cite{cantini_1989}, and have featured prominently in more recent literature \cite{hayashi_2022, rathjen_leigh_2010}. The steps in our proof can be summarized as follows:

\begin{enumerate}
\item Assume that $\mrm{VFM} \vdash A$ for an $\lnat$-sentence $A$.

\item Every theorem of $\mrm{VFM}$ is derivable in the corresponding infinitary sequent calculus $\mrm{VFM}^{\infty}$ (Lemma~\ref{lem:embedding_VFM}). In particular, we have $\mrm{VFM}^{\infty} \sststile{0}{\alpha} A$ with derivation length $\alpha < \varepsilon_0$.
Here, the derivation can be shown to be cut-free.

\item For every $\T$-positive sentence $A$, 
if $\mrm{VFM}^{\infty} \sststile{0}{\alpha} A$, 
then $\models^{\alpha} A$ (Theorem~\ref{thm:soundness}). Here, $\models^{\alpha}$ stands for some form of soundness, in which the truth predicate is interpreted as derivability in another sequent calculus. In particular, when $A$ is an $\lnat$-sentence, the soundness implies $\mathbb{N} \models A$, i.e., $A$ is true.

\item The above process is formalisable in a theory $\mrm{ID}^{*}_{1}$, which is proof-theoretically equivalent to $\mrm{KF}$. As a result, we obtain $\mrm{ID}^{*}_{1} \vdash A$ (Theorem~\ref{thm:upper-bound_VFM}).
\end{enumerate}

To begin with, we introduce the infinitary sequent calculus $\mrm{VFM}^{\infty}$ and show how VFM embeds into it. 

\subsection{The system VFM$^{\infty}$}

We formulate an infinitary derivation system $\mrm{VFM}^{\infty}$ as a Tait-style calculus, where each sentence is identified with its negation normal form.
A $sequent$ (denoted by $\varGamma, \varDelta, \dots$) is a finite set of $\lt$-sentences.
Given a sequent $\varGamma$ and sentences $A_1, \dots, A_n$, we write $\varGamma \cup \{ A_1, \dots, A_n \}$ as $\varGamma, A_1, \dots, A_n$.
For a sequent $\varGamma$, a natural number $k$, and an ordinal number $\alpha$, the predicate $\sststile{k}{\alpha} \varGamma$ means that, in $\mrm{VFM}^{\infty}$, $\varGamma$ is derivable with cut-rank $k$ and with derivation length $\alpha$. 

Here, we introduce some notation. 
For a formula $A(x)$ and closed terms $s$ and $t$,
the expression $A(t \simeq s)$ stands for any sentence $A(t)$ such that $t$ has the same value as $s$. 

The natural number $\mathrm{co}(A)$ denotes the logical complexity of $A$: 
\begin{itemize}
\item $\mathrm{co}(A) = \mathrm{co}(\neg A) = 0$, if $A$ is literal, i.e. an atomic formula or its negation,
\item $\mathrm{co}(A \land B) = \mathrm{co}(A \lor B) = \max \{ \mathrm{co}(A), \mathrm{co}(B) \} + 1$,
\item $\mathrm{co}(\forall x A(x)) = \mathrm{co}(\exists x A(x)) = \mathrm{co}(A(0)) + 1$.
\end{itemize} 

\begin{dfn}[System $\mrm{VFM}^{\infty}$]\label{dfn:vfminf}
The derivation system $\mrm{VFM}^{\infty}$ consists of Basic axioms, Basic rules and Truth principles, as displayed.

\begin{table}[H]
\begin{tabular}{l}
Basic axioms of $\mrm{VFM}^{\infty}$ \\
\hline \\
$(\mathsf{Ax.1})$ $\sststile{k}{\alpha} \varGamma, A,$ if $A$ is a true arithmetical literal.  \\
$(\mathsf{Ax.2})$ $\sststile{k}{\alpha} \varGamma, \neg \T(s), \T(t \simeq s)$. \\
\ \\
\hline \\
\end{tabular}
\end{table}

\begin{table}[H]
\begin{tabular}{ll}
Basic rules of $\mrm{VFM}^{\infty}$ \\
\hline \\
Assume: $\alpha_i < \alpha$ for each $i \in \mathbb{N}$ \\

\ \\

\infer[(\lor)]{\sststile{k}{\alpha} \varGamma, A_0\lor A_1}{\sststile{k}{\alpha_0} \varGamma, A_0, A_1} &
\infer[(\land)]{\sststile{k}{\alpha} \varGamma, A_0\land A_1}{\sststile{k}{\alpha_0} \varGamma, A_0 \ \ \ \sststile{k}{\alpha_1} \varGamma, A_1} \\

\ \\

\infer[(\exists)]{\sststile{k}{\alpha} \varGamma, \exists x A}{\sststile{k}{\alpha_0} \varGamma, A(i)} &
\infer[(\forall)]{\sststile{k}{\alpha} \varGamma, \forall x A}{\dots \sststile{k}{\alpha_i} \varGamma, A(i) \dots \ (i \in \mathbb{N})} \\

\ \\

\infer[(\mrm{cut})]{\sststile{k}{\alpha} \varGamma}{\sststile{k}{\alpha_0} \varGamma, A \ \ \ \sststile{k}{\alpha_1} \varGamma, \neg A}, where $\mathrm{co}(A) < k.$ \\
\ \\
\hline \\
\end{tabular}
\end{table}

\begin{table}[H]
\begin{tabular}{ll}
Truth principles of $\mrm{VFM}^{\infty}$ \\
\hline \\

Assume: $\alpha_i < \alpha$ for each $i \in \mathbb{N}$ \\

\ \\



\infer[(\T_=)]{\sststile{k}{\alpha} \varGamma, \T(t \simeq n \subdot{=} m)}{
\sststile{k}{\alpha_0} \varGamma, \mathrm{CT}(n) \land \mathrm{CT}(m) \land n^{\circ} = m^{\circ}} & 
\infer[(\T_{\neq})]{\sststile{k}{\alpha} \varGamma}{
\sststile{k}{\alpha_0} \varGamma, \mathrm{CT}(n) \land \mathrm{CT}(m) \land \T(n \subdot{=} m) \land n^{\circ} \neq m^{\circ}}  \\

\ \\

\infer[(\T_{\to})]{\sststile{k}{\alpha} \varGamma, \T (t \simeq m)}{
\sststile{k}{\alpha_0} \varGamma, \T (n) \ \ \ \sststile{k}{\alpha_1} \varGamma, \T (n \subdot{\to} m)} & 
\infer[(\T_{\forall})]{\sststile{k}{\alpha} \varGamma, \T (t \simeq \subdot{\forall} m n)}
{\dots \sststile{k}{\alpha_i} \varGamma, \T (n(i/m)) \dots \ (i \in \mathbb{N})} \\

\ \\

\infer[(\mrm{Rep})]{\sststile{k}{\alpha} \varGamma, \T (t \simeq \ulcorner \T (\dot{n}) \urcorner)}{\sststile{k}{\alpha_0} \varGamma, \T (n)} &
\infer[(\mathsf{Del})]{\sststile{k}{\alpha} \varGamma, \T (t \simeq n)}{\sststile{k}{\alpha_0} \varGamma, \T \ulcorner \T (\dot{n}) \urcorner} \\

\ \\

\infer[(\T_{\mrm{Cons}})]{\sststile{k}{\alpha} \varGamma, \T (t \simeq \ulcorner \neg (\T(\dot{n}) \land \T(\subdot{\neg}\dot{n})) \urcorner)}{} &
\infer[(\T_{\mrm{Comp}})]{\sststile{k}{\alpha} \varGamma, \T (t \simeq \ulcorner \mathrm{Sent}(\dot{n}) \to (\T(\dot{n}) \lor \T(\subdot{\neg}\dot{n})) \urcorner)}{} \\

\ \\

\infer[(\T_{\mrm{Norm}})]{\sststile{k}{\alpha}\varGamma, \T (t \simeq \ulcorner \T(\dot{n}) \to \mathrm{Sent}(\dot{n}) \urcorner)}{} &
\infer[(\T_{\pat})]{\sststile{k}{\alpha} \varGamma, \T(t \simeq n)}{\sststile{k}{\alpha_0} \varGamma,\mathrm{Ax}_{\pat}(n)} \\

\ \\

\infer[(\mrm{Cons})]{\sststile{k}{\alpha} \varGamma}{\sststile{k}{\alpha_0} \varGamma, \T(n) \ \ \ \sststile{k}{\alpha_0} \varGamma, \T(\subdot{\neg}n)} \\

\ \\

\hline \\
\end{tabular}
\end{table}
\end{dfn}
In the definition of $\mrm{VFM}^{\infty}$, the basic axioms and rules are used to embed $\mrm{PA}$. Each rule of the truth principles is necessary to embed the corresponding axiom of $\mrm{VFM}$ (see Lemma~\ref{lem:embedding_VFM}).


The following can be proved by the standard argument (cf.~\cite{hayashi_2022, rathjen_leigh_2010}).
Note that what is important for the cut eliminability to hold is that in each truth principle of $\mrm{VFM}^{\infty}$, the \emph{principal} (i.e., displayed in the conclusion) formula, if any, is of the form $\T(t)$ (see also Lemma~\ref{lem:cut-adm-I}). 
We also remark that such a principal formula is any formula of the form $\T(t \simeq s)$ for some $s$, not a particular one. This assumption is needed for establishing the substitution lemma (cf.~\cite[Proposition~3.4]{rathjen_leigh_2010}). 

Define $\omega_0(\alpha):=\alpha$, and $\omega_{n+1}(\alpha):=\omega^{\omega_{n}(\alpha)}$.

\begin{lemma}\label{lem:properties_VFM_inf}
\begin{description}
\item[(Substitution)] If $\sststile{k}{\alpha} \varGamma, A(s)$ and $s=t$ is true,
then $\sststile{k}{\alpha} \varGamma, A(t)$.
\item[(Weakening)] If $\sststile{k_0}{\alpha_0}\varGamma_0$, then $\sststile{k}{\alpha}\varGamma$ for any $k \geq k_0$, $\alpha \geq \alpha_0$, and $\varGamma \supseteq \varGamma_0$.

\item[(Cut-elimination)] If $\sststile{k}{\alpha}\varGamma$, then $\sststile{0}{\omega_k(\alpha)}\varGamma$.
\end{description}
\end{lemma}

Similarly to \cite[Proposition~9.5]{cantini_1989}, we can prove that $\mrm{VFM}^{\infty}$ derives all the consequences of $\mrm{VFM}$.
\begin{lemma}[Embedding]\label{lem:embedding_VFM}
Let $A$ be any $\lt$-sentence. If $\mathsf{VFM} \vdash A$, then $\mrm{VFM}^{\infty}\sststile{0}{\alpha}A$ for some $\alpha < \varepsilon_0$.
\end{lemma}

\begin{proof}
The proof is by induction on the derivation of $A$.
The axioms $\mrm{VF}1$ through $\mrm{VF}8$ are each derived by the corresponding rules of $\mrm{VFM}^{\infty}$.
In particular, $\mrm{VF}2$ is by $(\T_{\mrm{PAT}})$; $\mrm{VF}3$ is by $(\T_{\forall})$; $\mrm{VF}4$ is by $(\mrm{ReP})$ and $(\mrm{Del})$; $\mrm{VF}6$ is by $(\T_{\to})$; $\mrm{VF}7$ is by $(\T_{\mrm{Cons}})$ and $(\T_{\mrm{Comp}})$; $\mrm{VF}8$ is by $(\T_{\mrm{Norm}})$.
Here, recall that $\mrm{VF}5$ is redundant.
As for $\mrm{VF}1$, the first conjunct, i.e. the truth biconditional for equality, is derivable by $(\T_{=})$ and $(\T_{\neq})$. 
In addition, the right-to-left direction of the second conjunct is derived by $(\T_{\mrm{PAT}})$.
Finally, the converse direction is gained for each $m,n$:
\[ \
\infer[(\lor)]{\T(m \subdot{\neq} n) \to m^{\circ} \neq n^{\circ}}{\infer[(\mrm{Cons})]{\neg \T(m \subdot{\neq} n), m^{\circ} \neq n^{\circ}}{
\infer[(\T_{=})]{\neg \T(m \subdot{\neq} n), m^{\circ} \neq n^{\circ}, \T(m \subdot{=} n)}{\infer[\text{by logic}]{ \neg \T(m \subdot{\neq} n), m^{\circ} \neq n^{\circ}, \mrm{CT}(m) \land \mrm{CT}(n) \land m^{\circ} = n^{\circ}}{}} \ \ \ 
\infer[(\mrm{Ax}.2)]{\neg \T(m \subdot{\neq} n), m^{\circ} \neq n^{\circ}, \T(m \subdot{\neq} n)}{} 
}
},
\]

where the context $\varGamma := \neg \mrm{CT}(m), \neg \mrm{CT}(n)$ is omitted in each sequent.

The axioms and rules of $\mrm{PAT}$ are derived in the standard way.
\end{proof}

Our next step is to provide an interpretation of this infinitary system.

\subsection{Truth-as-provability interpretation for VFM}
\label{sec:truth-as-provability_VFM}
 
By the embedding lemma, the consistency of $\mrm{VFM}$ will follow from the consistency of $\mrm{VFM}^{\infty}$. Hence, our aim is to show the soundness of $\mrm{VFM}^{\infty}$.
For that purpose, we can follow Cantini's partial interpretation method for $\mrm{KF}$ (cf.~\cite{cantini_1989}), although ours is not \emph{asymmetric}, that is, we do not need to deal with negative occurrences of the truth predicate.

For each ordinal $\alpha$ and each $\T$-positive $\lt$-sentence $A$,
the relation $\models^{\alpha} A$, which stands for ``$A$ is satisfied at the level $\alpha$'', is inductively defined as follows:
\begin{itemize}
\item $\models^{\alpha}s=t$ $:\Leftrightarrow$ $s=t$ is true;

\item $\models^{\alpha}s \neq t$ $:\Leftrightarrow$ $s=t$ is false;

\item $\models^{\alpha} \T(t)$ $:\Leftrightarrow$ the value of $t$ is the G\"{o}del-number of some $\lt$-sentence $A$ and the $4$-ary relation $I(1;\alpha;\varepsilon_{\alpha}; A)$ holds;

\item $\models^{\alpha} A \land B$ $:\Leftrightarrow$ $\models^{\alpha} A$ and $\models^{\alpha} B$;

\item $\models^{\alpha} A \lor B$ $:\Leftrightarrow$ $\models^{\alpha} A$ or $\models^{\alpha} B$;

\item $\models^{\alpha} \forall x A(x)$ $:\Leftrightarrow$ $\models^{\alpha} A(n)$ for all $n$;

\item $\models^{\alpha} \exists x A(x)$ $:\Leftrightarrow$ $\models^{\alpha} A(n)$ for some $n$.
\end{itemize}
The $4$-ary relation $I(x;y;z;w)$, as the interpretation of the truth predicate, is defined in Definition~\ref{dfn:I}.
From the axioms of $\mrm{VFM}$, the conditions that $I(x;y;z;w)$ has to satisfy are determined.
Firstly, the truth predicate needs to derives all the axioms of $\mrm{PA}$, the axiom $\mrm{Cons}$ $(\forall x \neg (\T(x) \land \T(\subdot{\neg}x)))$, the  axiom $\mrm{Comp}$ ($\forall x (\mathrm{Sent}(x) \to \T(x) \lor \T(\subdot{\neg}x))$), and the $\T$-normality axiom $(\forall x (\T(x) \to \mathrm{Sent}(x)))$. As per the rules, it should be closed under Modus Ponens, $\omega$-rule, $\T$-Introduction (if $A$, then $\T\ulcorner A \urcorner$), and $\T$-Elimination (if $\T\ulcorner A \urcorner$, then $A$). Finally, the axiom $\mrm{VFM}1$ requires that $I$ is consistent.

In order to meet the above requirements, Cantini's $truth \mhyph as \mhyph provability$ interpretation \cite{cantini_1990} seems suitable.
So, we define the relation $I(x;y;z;w) \subseteq \{0,1\} \times \mrm{On} \times \mrm{On} \times \mrm{Seq}$ as a sequent calculus,
where $\mrm{On}$ is the set of ordinal numbers, and $\mrm{Seq}$ is the set of sequents.
The informal meaning of $I(i;\alpha;\beta;\varGamma)$ is that $\varGamma$ is derivable with applying $(\T \mhyph \mrm{Intro})$ at most $\alpha$-times and with the derivation height $\beta$. 
The index $i$ is used merely to distinguish derivations by whether the rule $(\mrm{Cons})$ or $(\mrm{Norm})$ (see Definition~\ref{dfn:I}) is used in it.
In particular, when $i=0$, it means that the derivation contains neither the rule $(\mrm{Cons})$ nor $(\mrm{Norm})$.

\begin{dfn}[Definition of $I$]\label{dfn:I}
The set $I \subseteq \{0,1\} \times \mrm{On} \times \mrm{On} \times \mrm{Seq}$ is defined to be the least fixed-point which is closed under the clauses below.
We write $I(x;y;z;w)$ instead of $\langle x,y,z,w \rangle \in I$.

Let $i \in \{0,1\}$, $\alpha, \beta \in \mrm{On}$, and $\varGamma \in \mrm{Seq}$.
Furthermore, we assume $\alpha_0 < \alpha$ and $\beta_0 < \beta$.

\begin{description}
\item[$(\mathrm{Ax}.1)$] $I(i; \alpha;\beta; \varGamma, s=t)$ holds, if $s=t$ is true.

\item[$(\mathrm{Ax}.2)$] $I(i; \alpha;\beta; \varGamma, s\neq t)$ holds, if $s\neq t$ is true.

\item[$(\mathrm{Ax}.3)$] $I(i; \alpha;\beta; \varGamma, \T(s), \neg \T(t \simeq s))$ holds.

\item[$(\lor)$] If $I(i;\alpha;\beta_0;\varGamma, A_0,A_1)$, then $I(i;\alpha;\beta;\varGamma, A_0 \lor A_1)$.

\item[$(\land)$] If $I(i;\alpha;\beta_0;\varGamma, A_0)$ and $I(i;\alpha;\beta_0;\varGamma, A_1)$, then $I(i;\alpha;\beta;\varGamma, A_0 \land A_1)$.

\item[$(\exists)$] If $I(i;\alpha;\beta_0;\varGamma, A(n))$ for some $n\in \nat$, then $I(i;\alpha;\beta;\varGamma, \exists x A(x))$.

\item[$(\forall)$] If $I(i;\alpha;\beta_0;\varGamma, A(n))$ for all $n\in \nat$, then $I(i;\alpha;\beta;\varGamma, \forall x A(x))$.

\item[$(\T \mhyph \mrm{Intro})$] If $I(i;\alpha_0;\beta_0;A)$, then 
$I(i;\alpha;\beta;\varGamma, \T (t \simeq \ulcorner A \urcorner))$.

\item[$(\mrm{Comp})$] $I(i;\alpha;\beta;\varGamma, \T(s \simeq \ulcorner A \urcorner), \T(t \simeq \ulcorner \neg A \urcorner))$ holds for every $\mathcal{L}_{\T}$-sentence $A$.

\item[$(\mrm{Cons})$] If $I(i;\alpha;\beta_0;\varGamma, \T(n))$ and $I(i;\alpha;\beta_0;\varGamma, \T(\subdot{\neg}n))$, then $I(1;\alpha;\beta;\varGamma)$.

\item[$(\mrm{Norm})$] If $I(i;\alpha;\beta_0;\varGamma, \T(n))$ and $\neg \mathrm{Sent}_{\lt}(n)$ holds, then $I(1;\alpha;\beta;\varGamma)$.

\end{description}
\end{dfn}


The reason for using the index $i$ in $I$ is explained as follows.
As remarked above, we have to assure the consistency of $I$. Although $I$ does not contain the cut rule, the consistency does not immediately follows, for the rules $(\mrm{Cons})$ and $(\mrm{Norm})$ break the subformula property of the derivation.
Thus, we also need to eliminate all applications of the rules $(\mrm{Cons})$ and $(\mrm{Norm})$. 
For this purpose, we use indices $i$ to record the information on these rules.
We also remark that the index $\alpha$ is necessary to eliminate these rules by transfinite induction on $\alpha$.


We can now state our goal, that is, the Soundness Theorem of $\mrm{VFM}$.
For a $\T$-positive sequent $\varGamma$, let $\models^{\alpha} \varGamma$ $:\Leftrightarrow$ $\models^{\alpha} A$ for some $A \in \varGamma$.

\begin{thm}[Soundness]\label{thm:soundness}
Let $\varGamma$ be any $\T$-positive sequent.
If $\mrm{VFM}^{\infty} \sststile{0}{\alpha} \varGamma$, then $\vDash^{\alpha} \varGamma$.
\end{thm}

The proof is given in the next section. As a consequence, the consistency of $\mathsf{VFM}$ follows:
\begin{corollary}[Consistency]
$\mrm{VFM}$ is consistent.
\end{corollary}

\begin{proof}
Assume, for contradiction, that $\mrm{VFM} \vdash 0=1$. Then, by the embedding lemma, we have $\mrm{VFM}^{\infty} \sststile{0}{\alpha}0=1$ for some $\alpha<\varepsilon_0$. Thus, the soundness theorem implies that $\models^{\alpha} 0=1$. However, by the definition of $\models^{\alpha}$, $0=1$ is satisfied at $\alpha$ only when $0=1$ is true, a contradiction. Therefore, $\mrm{VFM} \nvdash 0=1$.
\end{proof}

\subsection{Proof of the Soundness Theorem}
This subsection is dedicated to the proof of the Soundness Theorem (Theorem~\ref{thm:soundness}). 
As we remarked in the previous subsection, we have to show that $I$ is consistent and is closed under cut rule and $\T$-Elimination rule.

The following lemmata are proved similarly to for $\mrm{VFM}^{\infty}$.
\begin{lemma}[Substitution]\label{lem:subst_I}
If $I(i;\alpha;\beta;\varGamma,A(s))$ and $s=t$ is true, then $I(i;\alpha;\beta;\varGamma,A(t))$.
\end{lemma}

\begin{lemma}[Weakening]\label{lem:weak_I}
Assume $0 \leq i \leq j \leq 1$; $\alpha_0 \leq \alpha$; $\beta_0 \leq \beta$; and $\varGamma_0 \subseteq \varGamma$.
If $I(i;\alpha_0;\beta_0;\varGamma_0)$, then $I(j;\alpha;\beta;\varGamma)$.
\end{lemma}

In the case of $I$, we have cut admissibility instead of elimination.

\begin{lemma}[Cut-admissibility]\label{lem:cut-adm-I}
The following holds:

If $I(i;\alpha;\beta_0;\varGamma, A)$ and $I(i;\alpha;\beta_1;\varDelta,\neg A)$, then $I(i;\alpha;\omega_{\mrm{co}(A)}(\beta_0 \# \beta_1);\varGamma,\varDelta)$.
\end{lemma}

\begin{proof}
The proof is almost the same as for \cite[Lemma~3.6]{rathjen_leigh_2010}, but we must be careful not to increase the index $i$ and the $\T$-Intro-rank $\alpha$ through the cut rule.
Therefore, we observe several cases.

The proof is, as usual, by main-induction on $\mrm{co}(A)$ and sub-induction on $\beta_0 \# \beta_1$.
\begin{description}
\item[$\mrm{co}(A) = 0$.] We divide the cases by whether $A$ or $\neg A$ is principal in the last rule of the derivation. 
\begin{description}
\item[Both $A$ and $\neg A$ are principal.]
By symmetry, we can assume that $A$ is of the form $\T(t)$. 
Then, $\neg A$ can be principal only when $I(i;\alpha;\beta_1;\varDelta,\neg A)$ is an instance of $(\mrm{Ax}.3)$.
Thus, $\T(s) \in \varDelta$ for some $s$ with $s = t$.
Therefore, the claim $I(i;\alpha;\omega_{\mrm{co}(A)}(\beta_0 \# \beta_1);\varGamma,\varDelta)$ is directly obtained from $I(i;\alpha;\beta_0;\varGamma, A)$ by Lemma~\ref{lem:subst_I} and Lemma~\ref{lem:weak_I}.

\item[Either $A$ or $\neg A$ is not principal.]
By symmetry, we assume that $A$ is of the form $\T(t)$. 
Since the case where $\neg A$ is principal can be treated in the same way as for the above case, we can suppose not. 

As a special case, we further suppose that $I(i;\alpha;\beta_1;\varDelta,\neg A)$ is obtained by $(\T \mhyph \mrm{Intro})$.
Thus, we obtain the premise $I(i;\alpha_0;\beta_2;B)$ for some $\alpha_0 < \alpha$, $\beta_2 < \beta_1$, and an $\lt$-sentence $B$ such that $\T (s \simeq \ulcorner B \urcorner) \in \varDelta$.
Then, the claim $I(i;\alpha;\omega_{\mrm{co}(A)}(\beta_0 \# \beta_1);\varGamma,\varDelta)$ is directly obtained from the premise by $(\T \mhyph \mrm{Intro})$.

The other cases are immediate by the induction hypothesis, because $\neg A \equiv \neg\T(t)$ is contained in the premises of the last rule.
For example, assume that $I(i;\alpha;\beta_1;\varDelta,\neg A)$ is obtained by a two-premise rule $(R)$ from $I(i;\alpha;\beta_2;\varDelta_0,\neg A)$ and $I(i;\alpha;\beta_3;\varDelta_1,\neg A)$ for some $\beta_2,\beta_3 < \beta_1$ and $\varDelta_0, \varDelta_1 \in \mrm{Seq}$.
Then, applying the induction hypothesis to each of the premises, we have $I(i;\alpha;\omega_{\mrm{co}(A)}(\beta_0 \# \beta_2);\varGamma,\varDelta_0)$ and $I(i;\alpha;\omega_{\mrm{co}(A)}(\beta_0 \# \beta_3);\varGamma,\varDelta_1)$.
Thus, the rule $(R)$ gives the claim $I(i;\alpha;\omega_{\mrm{co}(A)}(\beta_0 \# \beta_1);\varGamma,\varDelta)$.

\end{description}

\item[$\mrm{co}(A) > 0$.] Similarly to the case~IIb in \cite[Lemma~3.6]{rathjen_leigh_2010}, The claim is established by using the main induction hypothesis on $\mrm{co}(A)$ (see also the proof of Lemma~\ref{lem:formal_cut-adm}).
\end{description}
\end{proof}

Next, we want to show the admissibility of $\T$-Elimination rule.

\begin{dfn}
A sequent $\varGamma$ is \emph{atomic}, if $\varGamma$ consists only of atomic sentences of $\lt$.
That is, $\varGamma$ contains only equations $s=t$ or truth predicates $\T(t)$. For an atomic sequent $\varGamma$, we define the \emph{disquotation} $\mrm{Disq}(\varGamma)$ to be the sequent 

\begin{equation*}
\{ s= t \ | \ s=t \in \varGamma \} \cup \{ A \ | \ \T (t) \in \varGamma \text{ for some closed term $t$ such that $t = \ulcorner A \urcorner$} \}
\end{equation*}

\end{dfn}

From the definition, if $\varGamma$ contains only equations $s=t$, then $\mrm{Disq}(\varGamma)$ is identical to $\varGamma$.

\begin{lemma}[Diquotation lemma]\label{lem:disq_I}
\begin{enumerate}
\item For an atomic sequent $\varGamma$ and $\alpha > 0$, assume that $I(i;\alpha;\beta;\varGamma)$. Then, $I(i;\alpha_0;\omega_n(\beta);\mrm{Disq}(\varGamma))$ holds for some $\alpha_0 < \alpha$ and some $n \in \mathbb{N}$.

\item In particular, when $\varGamma$ contains only equations $s=t$, we obtain $I(i;0;\varepsilon_{\beta};\varGamma)$ from the assumption $I(i;\alpha;\varepsilon_{\beta};\varGamma)$.
\end{enumerate}
\end{lemma}

\begin{proof}

Item 2 is immediate by transfinite induction on $\alpha$ from item 1.
In fact, for a sequent $\varGamma$ that contains only equations, we assume $I(i;\alpha;\varepsilon_{\beta};\varGamma)$.
Then, by item 1, we have 
$I(i;\alpha_0;\omega_n(\varepsilon_{\beta});\mrm{Disq}(\varGamma))$ for some $\alpha_0 < \alpha$ and some $n \in \mathbb{N}$.
Since $\mrm{Disq}(\varGamma) = \varGamma$ and $\omega_n(\varepsilon_{\beta}) = \varepsilon_{\beta}$,
it follows by the induction hypothesis that $I(i;0;\varepsilon_{\beta};\varGamma)$.

The proof of item~1 is by induction on the derivation length of $\varGamma$. 
We divide the cases by the last rule of the derivation.
\begin{description}
\item[$(\mathrm{Ax}.1)$] Assume that $I(i;\alpha;\beta;\varGamma)$ holds by $(\mathrm{Ax}.1)$. Thus, $\varGamma$ contains some true equation $s=t$. Then, since $s=t \in \mrm{Disq}(\varGamma)$, we have $I(i;\alpha_0;\omega_n(\beta);\mrm{Disq}(\varGamma))$ for any $n \in \mathbb{N}$ and any $\alpha_0 < \alpha$ by $(\mathrm{Ax}.1)$. 

\item[$(\T \mhyph \mrm{Intro})$]
Assume that $I(i;\alpha;\beta;\varGamma)$ is obtained from $I(i;\alpha_0;\beta_0;A)$ for some $\alpha_0 < \alpha$, some $\beta_0 < \beta$, and some $\T (t \simeq \ulcorner A \urcorner) \in \varGamma$.
Then, since $A \in \mrm{Disq}(\varGamma)$, we have $I(i;\alpha_0;\omega_n(\beta);\mrm{Disq}(\varGamma))$ by the Weakening Lemma for $I$. 

\item[$(\mrm{Comp})$] Assume that $I(i;\alpha;\beta;\varGamma)$ and $\{ \T (s \simeq \ulcorner A \urcorner), \T (t \simeq \ulcorner \neg A \urcorner) \} \subseteq \varGamma$.
Then, $\{ A, \neg A \} \subseteq \mrm{Disq}(\varGamma)$.
Since $I$ is closed under classical logic, we have $I(0;0;\omega;A,\neg A)$ by an easy induction. Thus, $I(i;\alpha_0;\omega_n(\beta);\mrm{Disq}(\varGamma))$ follows for any $n \geq 2$ by the Weakening Lemma for $I$.

\item[$(\mrm{Cons})$] Assume that $I(i;\alpha;\beta;\varGamma)$ is derived from $I(i_0;\alpha;\beta_0;\varGamma, \T(m))$ and $I(i_0;\alpha;\beta_0;\varGamma, \T(\subdot{\neg}m))$ for some $i_0 \leq i=1$, some $\beta_0 < \beta$, and some $m \in \mathbb{N}$. 
If $m$ does not denote any G\"{o}del-number of an $\lt$-sentence, then $\mrm{Disq}(\varGamma, \T(m)) = \mrm{Disq}(\varGamma)$, and thus the induction hypothesis and the Weakening Lemma yield $I(i;\alpha_0;\omega_n(\beta);\mrm{Disq}(\varGamma))$ for some $\alpha_0 < \alpha$ and some $n \in \mathbb{N}$.

Next, we assume that $m$ denotes $\ulcorner A \urcorner$ for some $\lt$-sentence.
Then, by the induction hypothesis, we have $I(i_0;\alpha_0;\omega_k(\beta_0);\mrm{Disq}(\varGamma), A)$ and $I(i_0;\alpha_1;\omega_l(\beta_0);\mrm{Disq}(\varGamma), \neg A)$ for some $\alpha_0,\alpha_1 < \alpha$ and for some $k,l \in \mathbb{N}$.
Thus, by Cut-admissibility for $I$, it follows that: 
\[
I(i;\max \{ \alpha_0, \alpha_1 \};\omega_{\mrm{co}(A)}(\omega_k(\beta_0) \# \omega_l(\beta_0));\mrm{Disq}(\varGamma)).
\]
Since $\max \{ \alpha_1, \alpha_2 \} < \alpha$ and $\omega_{\mathrm{co}(A)}(\omega_k(\beta_0) \# \omega_l(\beta_1)) < \omega_{\mathrm{co}(A)+ \max \{ k,l\}}(\beta)$, we can let $\alpha_0 := \max \{ \alpha_1, \alpha_2 \}$ and let $n := \mathrm{co}(A)+ \max \{ k,l\}$.
By Weakening for $I$, the conclusion is obtained for these $\alpha_0$ and $n$.

\item[$(\mrm{Norm})$] Assume $I(i;\alpha;\beta;\varGamma)$ is derived from $I(i_0;\alpha;\beta_0;\varGamma, \T(m))$ for some $\beta_0 < \beta$, where $\neg \mathrm{Sent}(m)$ is true.
Then, since $\mrm{Disq}(\varGamma, \T(t)) = \mrm{Disq}(\varGamma)$, the conclusion is obvious by the induction hypothesis. 

\end{description}
As $\varGamma$ is atomic, the other cases are impossible.

\end{proof}

\begin{remark}
Item~1 of the disquotation lemma informally says that if $\varGamma$ is derived with length $\beta_0$, then $\mrm{Disq}(\varGamma)$ is derivable with length $\omega_n(\beta)$ for some $n$.
For the proof of the Soundness Theorem to succeed, it is crucial that the derivation length of $\mrm{Disq}(\varGamma)$ can be kept below the least epsilon number larger than $\beta$ (see, e.g., the case of $\mrm{Del}$ in the proof of Theorem~\ref{thm:soundness} below).
This is made possible by the fact that no atomic sequent is derivable via an infinitary rule like the $\omega$-rule.
\end{remark}

A singleton $\{ \T\ulcorner A \urcorner \}$ is itself an atomic sequent, so the admissibility of $(\T \mhyph \mrm{Elim})$ follows from this lemma.
Thus, we can already establish that all the rules except $(\T_{\neq})$ and $(\mrm{Cons})$ are sound. On the other hand, as we remarked above, $(\T_{\neq})$ and $(\mrm{Cons})$ require the consistency of $I$.
Since the only obstacle to the consistency proof is the existence of $(\mrm{Cons})$ and $(\mrm{Norm})$ in $I$,
we want to eliminate applications of them from a given derivation.



\begin{lemma}[Elimination of $(\mrm{Cons})$ and $(\mrm{Norm})$]
Assume $I(i;0;\beta;\varGamma)$ for an atomic sequent $\varGamma$. Then, $I(0;0;\omega_n(\beta);\mrm{Disq}(\varGamma))$ holds for some $n \in \mathbb{N}$.
Therefore, if $\varGamma$ contains only equations $s=t$, we obtain $I(0;0;\omega_n(\beta);\varGamma)$.
\end{lemma}

\begin{proof}
Similarly to the proof of the Disquotation Lemma (Lemma~\ref{lem:disq_I}), the claim is established by induction on $\alpha$.
For example, the case $(\mrm{Cons})$ is proved in exactly the same way as for the same case in the proof of Disquotation Lemma.
Note that since the $\T \mhyph \mrm{Intro}$-rank $\alpha$ is $0$, the rule $\T \mhyph \mrm{Intro}$ is not used in the derivation of $\varGamma$.
\end{proof}

\begin{corollary}[Consistency of $I$]\label{cor:consistency_of_I}
No false equation $s=t$ is derivable in $I$, that is, if $I(i;\alpha;\beta;s=t)$, then $s=t$ is true. 
\end{corollary}

\begin{proof}
Assume $I(i;\alpha;\beta;s=t)$.
Then, by Weakening and the above lemmata, we have $I(0;0;\varepsilon_{\beta};s=t)$.
This means that $s=t$ is derived without $(\T \mhyph \mrm{Intro})$, $(\mrm{Cons})$, nor $(\mrm{Norm})$, which is possible only when $s=t$ is true.
\end{proof}

Finally, we give the proof of the Soundness theorem as promised.

\begin{lemma}[Persistency of $\models$]\label{lem:persist-I}
Let $\varGamma$ be a $\T$-positive sequent. 
If $\models^{\alpha_0} \varGamma$, then $\models^{\alpha} \varGamma$ for any $\alpha > \alpha_0$. 
\end{lemma}

\begin{proof}
The proof is similar to \cite[Lemma~9.8]{cantini_1989}.
\end{proof}

\begin{thm_soundness}
Let $\varGamma$ be any $\T$-positive sequent. If $\mrm{VFM}^{\infty} \sststile{0}{\alpha} \varGamma$, then $\vDash^{\alpha} \varGamma$.
\end{thm_soundness}

\begin{proof}
The proof is by induction on $\alpha$.
The cases are divided by the last rule of the derivation.

\begin{description}
\item[$(\T_{\neq})$] Assume that $\varGamma$ is derived as follows:
\begin{center} \
\infer[(\T_{\neq})]{\sststile{0}{\alpha} \varGamma}{
\sststile{0}{\alpha_0} \varGamma, \mathrm{CT}(n) \land \mathrm{CT}(m) \land \T(n \subdot{=} m) \land n^{\circ} \neq m^{\circ}}.
\end{center}
By the induction hypothesis, we have:
\[
\models^{\alpha_0} \varGamma, \mathrm{CT}(n) \land \mathrm{CT}(m) \land \T(n \subdot{=} m) \land n^{\circ} \neq m^{\circ}.
\]
For a contradiction, we assume:
\[
\models^{\alpha_0} \mathrm{CT}(n) \land \mathrm{CT}(m) \land \T(n \subdot{=} m) \land n^{\circ} \neq m^{\circ}.
\]
By the definition of $\models^{\alpha_0}$,
if an $\lnat$-sentence is satisfied at $\alpha_0$, then it should be true. Thus, we have
$\mathbb{N} \models \mathrm{CT}(n) \land \mathrm{CT}(m) \land n^{\circ} \neq m^{\circ}$.
On the other hand, $\models^{\alpha_0} \T(n \subdot{=} m)$ means $I(1;\alpha_0;\varepsilon_{\alpha_0};n^{\circ}=m^{\circ})$,
so $n^{\circ} = m^{\circ}$ has to be true by
Corollary~\ref{cor:consistency_of_I}. Therefore, we obtain a contradiction, and thus we get $\models^{\alpha_0} \varGamma$, which, by the persistency of $\models$ (Lemma~\ref{lem:persist-I}), implies $\models^{\alpha} \varGamma$.

\item[$(\T_{\to})$] 
We consider the following derivation:
\[
\infer[(\T_{\to})]{\sststile{k}{\alpha} \varGamma, \T(t \simeq m)}{\sststile{k}{\alpha_0} \varGamma, \T (n) \ \ \ \sststile{k}{\alpha_0} \varGamma, \T (n\subdot{\to}m)}.
\]
By persistency of $\models$, we can suppose that the induction hypotheses are $\models^{\alpha_0} \T (n)$ and $\models^{\alpha_0} \T (n\subdot{\to}m)$.
Thus, we have $I(1;\alpha_0; \varepsilon_{\alpha_0}; A)$ and $I(1;\alpha_0; \varepsilon_{\alpha_0}; A \to B)$ for some sentences $A, B$ such that $n = \ulcorner A \urcorner$ and $t^{\mathbb{N}} = m = \ulcorner B \urcorner$.
By Cut-admissibility for $I$ (Lemma~\ref{lem:cut-adm-I}), we obtain $I(1;\alpha_0; \omega_{\mathrm{co}(A)}(\varepsilon_{\alpha_0} \# \varepsilon_{\alpha_0}); B)$,
which, by Weakening for $I$, yields $\models^{\alpha} \T(t)$.

\item[$(\mrm{Cons})$] We consider the following derivation:
\[
\infer[(\mrm{Cons})]{\sststile{k}{\alpha} \varGamma}{\sststile{k}{\alpha_0} \varGamma, \T (n) \ \ \ \sststile{k}{\alpha_0} \varGamma, \T (\subdot{\neg}n)}.
\]
For a contradiction, we suppose both $\models^{\alpha_0} \T(n)$ and $\models^{\alpha_0} \T(\subdot{\neg}n)$.
Then, $n$ must denote some sentence $A$, and thus we have $I(1;\alpha_0; \varepsilon_{\alpha_0}; A)$ and $I(1;\alpha_0; \varepsilon_{\alpha_0}; \neg A)$.
By Cut-admissibility for $I$ (Lemma~\ref{lem:cut-adm-I}), we obtain $I(1;\alpha_0; \omega_{\mathrm{co}(A)}(\varepsilon_{\alpha_0} \# \varepsilon_{\alpha_0}); \emptyset)$, which contradicts Corollary~\ref{cor:consistency_of_I}.
Therefore, we get $\models^{\alpha_0} \varGamma$, which implies $\models^{\alpha} \varGamma$ by Lemma~\ref{lem:persist-I}.

\item[$(\T_{\forall})$] Use the fact that $I$ is closed under the $\omega$-rule. 

\item[$(\mrm{Rep})$] Use the fact that $I$ is closed under the rule $(\T \mhyph \mrm{Intro})$. 

\item[$(\mrm{Del})$] For $\varGamma = \varGamma', \T(t \simeq n)$, consider the following derivation:
\[
\infer[(\mrm{Del})]{\sststile{k}{\alpha} \varGamma', \T (t \simeq n)}{\sststile{k}{\alpha_0} \varGamma', \T \ulcorner \T (\dot{n}) \urcorner}.
\]
By induction hypothesis, we have $\models^{\alpha_0} \varGamma', \T \ulcorner \T (\dot{n}) \urcorner$. In particular, we can assume that
$\models^{\alpha_0} \T \ulcorner \T (\dot{n}) \urcorner$, and thus, we have $I(1;\alpha_0; \varepsilon_{\alpha_0};\T(n))$.
Then, by the consistency of $I$ (Corollary~\ref{cor:consistency_of_I}) and the Disquotation Lemma (Lemma~\ref{lem:disq_I}), $t$ and $n$ have to denote $\ulcorner A \urcorner$ for some $\lt$-sentence $A$.
Therefore, again by the Disquotation Lemma,
$I(1;\alpha_1; \omega_n(\varepsilon_{\alpha_0});A)$ holds for some $\alpha_1 < \alpha_0$ and some $n \in \mathbb{N}$.
Since $\omega_n(\varepsilon_{\alpha_0}) = \varepsilon_{\alpha_0} < \varepsilon_{\alpha}$,
Weakening for $I$ yields $I(1;\alpha;\varepsilon_{\alpha};A)$, which means $\models^{\alpha} \T(t)$, as required.
\end{description}
The other rules are similarly treated.
\end{proof}

\subsection{Formalising the consistency proof}
Based on the consistency proof of $\rm{VFM}$ in the previous subsection, we now want to show that a lower bound of $\mrm{VFM} $ given in Section~\ref{subsec:VFM-lower-bound} is indeed exact, that is, $\vert \mrm{VFM} \vert \leq \vert \mrm{RT}_{< \varepsilon_0} \vert$. 
A natural idea, then, would be to formalise our consistency proof of $\mrm{VFM}$ in $\mrm{RT}_{< \varepsilon_0}$ or another theory equivalent to $\mrm{RT}_{< \varepsilon_0}$, as Cantini \cite{cantini_1989} did for $\mrm{KF}$.
This task, however, is not so simple, because the derivation system $I$, defined in our proof, uses derivations whose  derivation height exceeds $\varepsilon_0$. 
Since $I$ has the $\omega$-rule, its derivation is generally not recursive and therefore $\mrm{RT}_{< \varepsilon_0}$ seems insufficient to formalise it.\footnote{In fact, $\mrm{RT}_{< \varepsilon_{\varepsilon_0}}$ is enough.}
To overcome this difficulty, we use an expressively rich system $\mrm{ID}^*_1$,
the theory of positive induction, which is known to still be conservative over $\mrm{RT}_{< \varepsilon_0}$ (cf. \cite{afshari_rathjen_2010,arai_2018,probst_2006}).
Thus, firstly we define $\mrm{ID}^*_1$.

For a new unary predicate $Q(x)$, let $\mathcal{L}_{Q} := \lnat \cup \{ Q(x) \}$.
Then, the language $\lpos$ of $\mathsf{ID}^*_1$ is defined to be an expansion of $\lnat$ with new unary predicate symbols $\mrm{I}_A$, where $A : \equiv A(u, Q)$ is any $\mathcal{L}_Q$-formula such that only $u$ may occur free and every occurrence of $Q$ is positive in $A$.
The predicate $\mrm{I}_A$ is intended to denote a \emph{quasi} least fixed-point of the positive operator $A$.
We let $\mrm{PA}_{\lpos}$ be $\mrm{PA}$ formulated over $\lpos$.
In particular, $\mrm{PA}_{\lpos}$ has the induction schema for all $\lpos$-formulae.

\begin{dfn}
The $\lpos$-theory $\mrm{ID}^*_1$ consists of $\mrm{PA}_{\lpos}$ with the following:
\begin{description}
\item[($\mrm{I}_A.1$)] $\forall u [A(u, \mrm{I}_A) \leftrightarrow \mrm{I}_A(u)]$.
\item[($\mrm{I}_A.2$)] $\forall u [A(u, F) \to F(u)] \to \forall u [\mrm{I}_A(u) \to F(u)]$,

where $F(x)$ is any $\lpos$-formula in which the new predicates $\mrm{I}_B$ $(B \in \mathcal{L}_Q)$ occur only positively.
\end{description}
\end{dfn}

As a note, the fixed-point theory $\widehat{\mrm{ID}}_1$ does not have the axioms ($\mrm{I}_A.2$),
whereas the theory $\mrm{ID}_1$ of positive inductive definition admits all $\lpos$-formulae $F$ in the axioms ($I_A.2$). The intermediate theory $\mathsf{ID}^*_1$ is known to be significantly weaker than $\mrm{ID}_1$:

\begin{fact}[\cite{afshari_rathjen_2010, arai_2018, probst_2006}]
$\mrm{ID}^{*}_{1}$ is arithmetically conservative over the fixed-point theory $\widehat{\mrm{ID}}_1$.
Therefore, $\vert \mrm{ID}^{*}_{1} \vert \equiv \varphi \varepsilon_0 0$.
\end{fact}

Our strategy for formalising the consistency proof of $\mrm{VFM}$ is to express the relation $I$ within $\mrm{ID}^{*}_{1}$.
Since the relation $I$ was positively defined in Definition~\ref{dfn:I}, we can construct a fixed point even in $\widehat{\mrm{ID}}_1$.
But as we remarked above, we cannot rely on arguments by transfinite induction beyond $\varepsilon_0$.
Instead, we make use of the axiom ($\mrm{I}_A.2$), which allows induction on the length of the derivation.

It is known that basic set-theoretic notions and operations are formalisable in $\mrm{PA}$, so we shall use the same notation as in the previous section. 
Let $\mathrm{Seq}(x)$ be a unary predicate meaning that $x$ is the code of a sequent; $\{ x_0. x_1, \dots, x_n \}$ denotes the code of a finite set consisting of $x_0, \dots, x_n$; $x \cup y$ denotes the code of the union of sets $x$ and $y$. 
Similarly to the original definition of $I$, we suppress angle brackets, so we mean  $Q(\langle x,y,z \rangle)$ by $Q(x;y;z)$.

\begin{dfn}[Formalised $\mrm{I}$]
The $\mathcal{L}_{Q}$-formula $A^I(u,Q)$ is defined to be 
\[
(u)_0 \leq 1 \land (u)_1 \in \mrm{OT} \land (u)_2 \in \mathrm{Seq} \land (\star),
\]
where the formula $(\star)$ is the disjunction of the following:
\begin{description}
\item[$(\mathrm{Ax}.1)$] $\exists s,t (s^{\circ} = t^{\circ} \land (s\subdot{=} t) \in (u)_2)$

\item[$(\mathrm{Ax}.2)$] $\exists s,t (s^{\circ} \neq t^{\circ} \land (s\subdot{\neq} t) \in (u)_2)$

\item[$(\mathrm{Ax}.3)$] $\exists s,t (s^{\circ} = t^{\circ} \land \{ \subdot{\T}s, \subdot{\neg}(\subdot{\T}t) \} \in (u)_2)$

\item[$(\lor)$] $\exists \varGamma' \exists x,y [(u)_2 = \varGamma' \cup \{ x \subdot{\lor} y \} \land Q((u)_0;(u)_1; \varGamma' \cup  \{ x,y \})]$

\item[$(\land)$] $\exists \varGamma' \exists x,y [(u)_2 = \varGamma' \cup \{ x \subdot{\land} y \} \land Q((u)_0;(u)_1; \varGamma' \cup  \{ x \}) \land Q((u)_0;(u)_1; \varGamma' \cup  \{ y \} )]$

\item[$(\exists)$] $\exists \varGamma' \exists v \exists x [(u)_2 = \varGamma' \cup \{ \subdot{\exists}vx \} \land \exists n (Q((u)_0;(u)_1;\varGamma' \cup \{ x(n/v) \} ))]$

\item[$(\forall)$] $\exists \varGamma' \exists v \exists x [(u)_2 = \varGamma' \cup \{ \subdot{\forall}vx \} \land \forall n (Q((u)_0;(u)_1;\varGamma' \cup \{ x(n/v) \} ))]$

\item[$(\T \mhyph \mrm{Intro})$] $\exists \varGamma' \exists t [t^{\circ} \in \mathrm{Sent} \land \subdot{\T}t \in (u)_2 \land \exists \alpha_0 < (u)_1 (Q((u)_0; \alpha_0; \{ t^{\circ} \} ))]$

\item[$(\mrm{Comp})$] $\exists s,t [s^{\circ} \in \mathrm{Sent} \land \subdot{\neg} (s^{\circ}) = t^{\circ} \land \{ \subdot{\T}s, \subdot{\T}t \} \subseteq (u)_2]$

\item[$(\mrm{Cons})$] $\exists j \leq (u)_0 \exists n [(u)_0=1 \land Q(j; (u)_1; (u)_2 \cup \{ \subdot{\T}n \} ) \land Q(j; (u)_1; (u)_2 \cup \{ \subdot{\T}(\subdot{\neg}n) \} )]$

\item[$(\mrm{Norm})$] $\exists j \leq (u)_0 \exists n [(u)_0 = 1 \land Q(j; (u)_1; (u)_2 \cup \{ \subdot{\T}n \} ) \land \neg \mathrm{Sent}(n^{\circ})]$

\end{description}
Since $A^I(u,Q)$ is a $Q$-positive formula, $\widehat{\mrm{ID}}_1$ has a fixed point of $A^I$.
So, we take such a predicate symbol $\mrm{I}(u)$:
\[
\widehat{\mrm{ID}}_1 \vdash \forall u (\mrm{I}(u) \leftrightarrow A^I(u, \mrm{I})).
\]
\end{dfn}


Let $\varGamma \in {\rm{Seq}}_v$ mean that $\varGamma$ is a set of formulas in which only $v$ may occur free.
For such $\varGamma$, we define $\varGamma[t/v]$ to be the result of replacing $v$ by $t$ simultaneously for all formulae of $\varGamma$. 
Thus, if $x \in \varGamma$, then we have $x(t/v) \in \varGamma[t/v]$.

\begin{lemma}[Formalised Substitution]
Let a formula $F(u)$ be the following:
\[
F(u):\equiv \forall x \forall v \forall s,t \forall \varGamma \in {\rm{Seq}}_v [(u)_2 = \varGamma[t/v] \land s^{\circ} = t^{\circ} \to \mrm{I}((u)_0; (u)_1; \varGamma[s/v] )].
\]
Then, $\mathsf{ID}^{*}_{1} \vdash \forall u (\mrm{I}(u) \to F(u)) $.
In particular, the Substitution Lemma of the familiar form is obtained:
\[
\forall \varGamma \in \mrm{Seq}[\mrm{I}(i;\alpha;\varGamma \cup \{ x(t/v) \}) \land s^{\circ} = t^{\circ} \to \mrm{I}(i;\alpha;\varGamma \cup \{ x(s/v) \})].
\]
\end{lemma}

\begin{proof}
The informal meaning of the formula $F$ should be clear: if the third element $(u)_2$ of the sequence $u$ is a sequent $\varGamma[t/v]$, then every sentence $x(t/v)$ in $(u)_2$ can be simultaneously replaced by $x(s/v)$ for any closed term $s$ such that $s^{\circ} = t^{\circ}$. Moreover, the values of $(u)_0$ and $(u)_1$ remain unchanged.
Note also that $F$ so defined is $\mrm{I}$-positive, so we can use the axiom $(\mrm{I}_{A^{I}}.2)$ to prove the claim.
Thus, it suffices to show $\forall u [A^I(u,F) \to F(u)]$.

Taking any $u$ and assuming $A^I(u,F)$, we prove $F(u)$. 
So, we further take any $x,v,s,t,\varGamma$ such that $(u)_2 = \varGamma[t/v] $ and $s^{\circ} = t^{\circ}$.
Then, we have to show $\mrm{I}((u)_0; (u)_1; \varGamma[s/v] )$.
The proof is divided by cases according to which disjunct of $A^I(u,F)$ holds.

\begin{description}
\item[$(\mrm{Ax}.1)$] In this case, $(u)_2 = \varGamma[t/v]$ contains some equation $s_1\subdot{=} t_1$ with $s_1^{\circ} = t_1^{\circ}$. 
Now we write $s_1 \equiv (s_2(t/v))$ and $t_1 \equiv (t_2(t/v))$.
Since $s^{\circ} = t^{\circ}$ is true, so is $(s_2(t/v))^{\circ} = (s_2(s/v))^{\circ}$, which can be verified by formal induction on $s_2$. 
Similarly, $(t_2(t/v))^{\circ} = (t_2(s/v))^{\circ}$ holds.
Therefore, $s_2(s/v) \subdot{=} t_2(s/v)$ is true and is contained in $\varGamma[s/v]$.
Thus, we have $A^I(\langle  (u)_0, (u)_1, \varGamma[s/v] \rangle)$ by $(\mrm{Ax}.1)$, which implies $\mrm{I}((u)_0; (u)_1; \varGamma[s/v] )$ by $(\mrm{I}_{A^I}.1)$.


\item[$(\lor)$] 
In this case, $(u)_2 = \varGamma[t/v]$ contains $(x \subdot{\lor} y)(t/v) $ for some $x,y$. 
Then, we may assume that the premise of $(u)_2$ is $\varGamma[t/v] \cup  \{ (x(t/v),y(t/v) \}$.
Intuitively, we consider the following derivation:
\begin{center} \
\infer[(\lor)]{\varGamma[t/v]}{\varGamma[t/v],x(t/v),y(t/v)}    
\end{center}
Then, by $A^I(u,F)$, we get $\mrm{I}((u)_0; (u)_1; \varGamma[s/v],x(s/v),y(s/v))$. Thus, $\mrm{I}((u)_0; (u)_1; \varGamma[s/v])$ follows by $(\mrm{I}_{A^I}.1)$.

\end{description}
The other cases are similar.
\end{proof}

A formalised version of Weakening is proved as well.
\begin{lemma}[Formalised Weakening]
Let a formula $F(u)$ be the following:
\[
F(u) :\equiv \forall i \in \{ (u)_0, 1 \} \forall \alpha \geq (u)_1 \forall \varGamma \supseteq (u)_2 (\mrm{I}(i;\alpha; \varGamma )).
\]
Then, $\mrm{ID}^{*}_{1} \vdash \forall u (\mrm{I}(u) \to F(u)) $.
\end{lemma}

Before the full cut-admissibility, we first establish the atomic case.

\begin{lemma}[Formalised $\T$-Cut]\label{lem:formal_T-Cut}
Taking any $i \leq 1$, $\alpha \in \mrm{OT}$, $\varGamma \in \mathrm{Seq}$, and $t \in \mathrm{CT}$,
let a formula $F(u)$ (with parameters $i,\alpha,\varGamma,t$) be the following:
\[
F(u):\equiv 
\mrm{I}(u) \land \mrm{I}(\max \{i,(u)_0 \}; \max \{\alpha,(u)_1 \}; (u)_2 [ \varGamma / \{ \subdot{\neg}(\subdot{\T}t) \}] ),
\]
where $(u)_2 [ \varGamma / \{ \subdot{\neg}(\subdot{\T}t) \}]$ is the result of eliminating $\subdot{\neg}(\subdot{\T}t)$ contained in $(u)_2$ and instead adding every member of $\varGamma$; if $\subdot{\neg}(\subdot{\T}t)$ is not contained in $(u)_2$, then it just returns $(u)_2$.

Then, $\mrm{ID}^{*}_{1} \vdash \mrm{I}(i; \alpha; \varGamma \cup \{ \subdot{\T}t \} ) \to \forall u (\mrm{I}(u) \to F(u))$.
In particular, Cut-admissibility for $\T$ is obtained:
\[
\mrm{ID}^{*}_{1} \vdash \mrm{I}(i; \alpha; \varGamma \cup \{ \subdot{\T}t \} ) \land \mrm{I}(i; \alpha; \varDelta \cup \{ \subdot{\neg}(\subdot{\T}t) \} ) \to \mrm{I}(i; \alpha; \varGamma, \varDelta ).
\]
\end{lemma}

\begin{proof}
Suppose $\mrm{I}(i; \alpha; \varGamma \cup \{ \subdot{\T}t \} )$.
Since $F(u)$ is $\mrm{I}$-positive (with the parameters $i,\alpha,\varGamma,t$), it suffices by $(\mrm{I}_{A^I}.2)$ to show that $F$ is closed under $A^I$: $\forall u [A^I(u,F) \to F(u)]$.
So, taking any $u$ such that $A^I(u,F)$, we want to derive $F(u)$. The cases are divided by which disjunct of $A^I(u,F)$ holds.
\begin{description}
\item[$(\mathrm{Ax}.1)$] Assume that there exist $s_0,s_1 \in \mathrm{CT}$ such that $s_0^{\circ} = s_1^{\circ}$ and $(s_0 \subdot{=} s_1) \in (u)_2$. 
Then, $\mrm{I}(u)$ is obvious.
Moreover, it is clear that $(s_0 \subdot{=} s_1) \in (u)_2 [ \varGamma / \{ \subdot{\neg}(\subdot{\T}t) \}]$, so it follows that $\mrm{I}(\max \{i,(u)_0 \}; \max \{\alpha,(u)_1 \}; (u)_2 [ \varGamma / \{ \subdot{\neg}(\subdot{\T}t) \}] )$.
Therefore, we have $F(u)$.

\item[$(\mathrm{Ax}.3)$] Assume that there exist $s_0,s_1 \in \mathrm{CT}$ such that $s_0^{\circ} = s_1^{\circ}$ and $\{ \subdot{\T}s_0, \subdot{\neg} (\subdot{\T}s_1) \} \in (u)_2$.
If $s_1 \neq t$, then $\{ \subdot{\T}s_0, \subdot{\neg} (\subdot{\T}s_1) \} \in (u)_2 [ \varGamma / \{ \subdot{\neg}(\subdot{\T}t) \}]$, and thus we have $F(u)$. Therefore, we can suppose $s_1 = t$. Then, since $s_0^{\circ} = t^{\circ}$, the supposition $\mrm{I}(i; \alpha; \varGamma \cup \{ \subdot{\T}t \} )$ implies by Formalised Substitution that $\mrm{I}(i; \alpha; \varGamma \cup \{ \subdot{\T}s_0 \} )$.
As $\varGamma \cup \{ \subdot{\T}s_0 \} \subseteq  (u)_2 [ \varGamma / \{ \subdot{\neg}(\subdot{\T}t) \}]$, we can conclude $F(u)$ by Formalised Weakening.

\item[$(\lor)$] Assume $(u)_2 = \varDelta \cup \{ x \subdot{\lor} y \}$ for some $\varDelta \in \mathrm{Seq}$ and $x,y \in \mathrm{Sent}$. Then, the induction hypothesis is as follows:
\[
\mrm{I}((u)_0;(u)_1;\varDelta \cup \{ x,y \}) \ \land \
\mrm{I}(\max \{i,(u)_0 \};\max \{\alpha,(u)_1 \};(\varDelta \cup \{ x, y \})[\varGamma / \{ \subdot{\neg}(\subdot{\T}t) \}]).
\] 

Firstly, $\mrm{I}(u)$, the first conjunct of $F(u)$, is clear by $\mrm{I}((u)_0;(u)_1;\varDelta \cup \{ x,y \})$, because $\mrm{I}$ is closed under the rule $(\lor)$. 
Also, if $\subdot{\neg}(\subdot{\T}t) \notin (u)_2$, then $(u)_2 [ \varGamma / \{ \subdot{\neg}(\subdot{\T}t) \}] = (u)_2$ and thus the second conjunct of $F(u)$ follows from $\mrm{I}(u)$ by Formalised Weakening.
Thus, we  can assume $\subdot{\neg}(\subdot{\T}t) \in (u)_2$.
Since $(\subdot{\neg}(\subdot{\T}t)) \neq (x \subdot{\lor} y)$, this implies $\subdot{\neg}(\subdot{\T}t) \in \varDelta$.
Then, we have: 
\[
(\varDelta \cup \{ x, y \})[\varGamma / \{ \subdot{\neg}(\subdot{\T}t) \}] \subseteq (u)_2[\varGamma / \{ \subdot{\neg}(\subdot{\T}t) \}] \cup \{ x,y \}.
\]
Therefore, by Formalised Weakening, the second conjunct of the induction hypothesis implies:
\[
\mrm{I}(\max \{i,(u)_0 \};\max \{\alpha,(u)_1 \};(u)_2[\varGamma / \{ \subdot{\neg}(\subdot{\T}t) \}] \cup \{ x,y \}).
\]
Since $x \subdot{\lor} y \in (u)_2[\varGamma / \{ \subdot{\neg}(\subdot{\T}t) \}]$, we have the second conjunct of $F(u)$ by the rule $(\lor)$:  
\[
\mrm{I}(\max \{i,(u)_0 \};\max \{\alpha,(u)_1 \};(u)_2[\varGamma / \{ \subdot{\neg}(\subdot{\T}t) \}]).
\]
To conclude, we obtain $F(u)$.
\end{description}
The other cases are similarly proved.
\end{proof}

The following, which is required for the cut-admissibility for (in)equality, is shown similarly to Formalised $\T$-Cut.
\begin{lemma}[Formalised Falsity-elimination]\label{lem:formal_falsity}
Take any $i \leq 1$, $\alpha \in \mrm{OT}$, $\varGamma \in \mathrm{Seq}$, and $s,t \in \mathrm{CT}$. Then,
$\mrm{ID}^{*}_{1} \vdash  \mrm{I}(i;\alpha;\varGamma \cup \{ s \subdot{=} t \}) \land s^{\circ} \neq t^{\circ} \to \mrm{I}(i;\alpha;\varGamma)$.
\end{lemma}

We now state full cut-admissibility:
\begin{lemma}[Formalised Cut-admissibility]\label{lem:formal_cut-adm}
Take any $i \leq 1$, $\alpha \in \mrm{OT}$, $\varGamma, \varDelta \in \mathrm{Seq}$, and $x \in \mathrm{Sent}$. Then, $\mrm{ID}^{*}_{1}$ derives the following:
\[
\mrm{I}(i;\alpha;\varGamma \cup \{ x \} ) \land \mrm{I}(i;\alpha;\varDelta \cup \{ \subdot{\neg}x \} ) \to
\mrm{I}(i;\alpha;\varGamma \cup \varDelta ).
\]
\end{lemma}

\begin{proof}
The proof is by formal induction on the logical complexity of $x$.
The case where $x$ is of the form $\subdot{\T}t$ for some $t \in \mathrm{CT}$ is by Lemma~\ref{lem:formal_T-Cut}. If $x$ is an equation or a negated equation, then the conclusion follows from Lemma~\ref{lem:formal_falsity}.
As to the inductive steps, by symmetry, it suffices to consider the case where $x$ is either conjunctive or universal. Thus, we show the following subsidiary lemmata.
\end{proof}

\begin{lemma}[Formalised Inversion]\label{lem:formal-inv}
In $\mrm{ID}^{*}_{1}$, take any $i \leq 1$, $\alpha \in \mrm{OT}$, $\varGamma \in \mathrm{Seq}$,  $x,y \in \mathrm{Sent}$, $v \in \mathrm{Var}$, and $z \in \mathrm{Form}$.
Then, the following are derived:
\begin{description}
\item[$(\land \mhyph \mrm{Inv})$] $\mrm{I}(i; \alpha; \varGamma \cup \{ x \subdot{\land} y \} ) \to \mrm{I}(i; \alpha; \varGamma \cup \{ x \} ) \land \mrm{I}(i; \alpha; \varGamma \cup \{ y \} )$


\item[$(\forall \mhyph \mrm{Inv})$] $\mrm{I}(i; \alpha; \varGamma \cup \{ \subdot{\forall}vz \} ) \to \forall n (\mrm{I}(i; \alpha; \varGamma \cup \{ z(n/v) \} ))$
\end{description}
\end{lemma}

\begin{lemma}[Formalised $\land$-Cut]
We argue in $\mrm{ID}^{*}_{1}$.
Taking any $i \leq 1$, $\alpha \in \mrm{OT}$, and $x,y \in \mathrm{Sent}$,
we assume that Formalised Cut-admissibility holds for every sentence $z$ such that $\mathrm{co}(z) < \mathrm{co}(x \subdot{\land}y)$:
\[
\forall z \big(\mrm{co}(z) < \mrm{co}(x \subdot{\land} y) \to \forall \varGamma, \varDelta [\mrm{I}(i;\alpha;\varGamma \cup \{ z \} ) \land \mrm{I}(i;\alpha;\varDelta \cup \{ \subdot{\neg}z \} ) \to
\mrm{I}(i;\alpha;\varGamma \cup \varDelta )]\big).
\]
Then, Formalised Cut-admissibility for $x \subdot{\land}y$ holds for any $\varGamma, \varDelta$:
\[
\mrm{I}(i;\alpha;\varGamma \cup \{ x \subdot{\land}y \} ) \land \mrm{I}(i;\alpha;\varDelta \cup \{ \subdot{\neg}(x \subdot{\land}y) \} ) \to
\mrm{I}(i;\alpha;\varGamma \cup \varDelta ).
\]

\end{lemma}

\begin{proof}
Taking any $i \leq 1$, $\alpha \in \mrm{OT}$, and $\varGamma \in \mathrm{Seq}$,
let a formula $F(u)$ be the following:
\[
F(u):\equiv 
\mrm{I}(u) \land \mrm{I}(\max \{i, (u)_0 \}; \max \{\alpha, (u)_1 \}; (u)_2 [ \varGamma / \{ \subdot{\neg}(x\subdot{\land} y) \}] ),
\]
where $(u)_2 [ \varGamma / \{ \subdot{\neg}(x\subdot{\land} y) \}]$ is the result of eliminating $\subdot{\neg}(x\subdot{\land} y)$ occurring in $(u)_2$ and instead adding every member of $\varGamma$. If $\subdot{\neg}(x\subdot{\land} y)$ is not contained in $(u)_2$, then it just returns $(u)_2$.

Then, we show that $\mrm{ID}^{*}_{1}$ derives $\mrm{I}(i; \alpha; \varGamma \cup \{ x\subdot{\land} y \} ) \to \forall u (\mrm{I}(u) \to F(u))$, which clearly implies Formalised Cut-admissibility for $x \subdot{\land} y$ by putting $u := \langle i;\alpha;\varDelta \cup \{ \subdot{\neg} (x \subdot{\land} y) \}\rangle$.
Since $F(u)$ is $\mrm{I}$-positive (with parameters $i,\alpha,\varGamma$), it suffices to derive $\forall u (A^I(u,F) \to F(u))$, assuming $\mrm{I}(i; \alpha; \varGamma \cup \{ x\subdot{\land} y \} )$.
So, taking any $u$ such that $A^I(u,F)$, we prove $F(u)$.

As the most important case, we suppose that $A^I(u,F)$ holds by the rule $(\lor)$. In particular, we assume that $(u)_2 = \varDelta' \cup \{ \subdot{\neg} (x \subdot{\land} y ) \}$ for some $\varDelta' \in \mathrm{Seq}$ and the induction hypothesis is the conjunction of the following:
\begin{itemize}
\item $\mrm{I}((u)_0; (u)_1; \varDelta' \cup \{ \subdot{\neg}x,\subdot{\neg}y \})$,
\item $\mrm{I}((u)_0;(u)_1; (\varDelta' \cup \{ \subdot{\neg}x,\subdot{\neg}y \})[\varGamma / \{ \subdot{\neg}(x\subdot{\land} y) \}])$.
\end{itemize}
Firstly, $\mrm{I}(u)$ is obvious by the first conjunct of the induction hypothesis and the rule $(\lor)$.
Secondly, the second conjunct of the induction hypothesis is equivalent to $\mrm{I}((u)_0;(u)_1; \varDelta' [\varGamma / \{ \subdot{\neg}(x\subdot{\land} y) \}] \cup \{ \subdot{\neg}x,\subdot{\neg}y \})$. 
On the other hand, the assumption $\mrm{I}(i; \alpha; \varGamma \cup \{ x\subdot{\land} y \} )$ implies the following by Formalised Inversion (Lemma~\ref{lem:formal-inv}):
\[
\mrm{I}(i; \alpha; \varGamma \cup \{ x \} )
\land \mrm{I}(i; \alpha; \varGamma \cup \{ y \} ).
\]
Thus, by applying Formalised Cut-admissibility for $x$ to the second conjunct of the induction hypothesis, we have:
\[
\mrm{I}(\max \{i, (u)_0 \}; \max \{\alpha, (u)_1 \}; \varDelta' [\varGamma / \{ \subdot{\neg}(x\subdot{\land} y) \}] \cup \{ \subdot{\neg}y \} \cup \varGamma).
\]
Therefore, by Formalised Cut-admissibility for $y$, we obtain further:
\[
\mrm{I}(\max \{i, (u)_0 \}; \max \{\alpha, (u)_1 \}; \varDelta' [\varGamma / \{ \subdot{\neg}(x\subdot{\land} y) \}] \cup \varGamma).
\]
Since $\varDelta' [\varGamma / \{ \subdot{\neg}(x\subdot{\land} y) \}] \cup \varGamma \subseteq (u)_2[\varGamma / \{ \subdot{\neg}(x\subdot{\land} y) \}]$, the second conjunct of $F(u)$ follows by Formalised Weakening:
\[
\mrm{I}(\max \{i, (u)_0 \}; \max \{\alpha, (u)_1 \}; (u)_2 [ \varGamma / \{ \subdot{\neg}(x\subdot{\land} y) \}] ).
\]
To summarise, $F(u)$ is obtained, as required.
\end{proof}

Formalised $\forall$-Cut is admissible as well, which completes the proof of Lemma~\ref{lem:formal_cut-adm}:
\begin{lemma}[Formalised $\forall$-Cut]
We argue in $\mrm{ID}^{*}_{1}$.
Taking any $(\subdot{\forall}vx) \in \mathrm{Sent}$,
we assume that Formalised Cut-admissibility holds for every sentence $z$ such that $\mathrm{co}(z) < \mathrm{co}(\subdot{\forall}vx)$.
Then, Formalised Cut-admissibility also holds for $\subdot{\forall}vx$.
\end{lemma}

Let a primitive recursive predicate $\mathrm{AtSeq}(x)$ mean that the G\"{o}del-number $x$ of some sequent consists only of the G\"{o}del-number of atomic sentences. 
Moreover, we define a primitive recursive function $\mathrm{Disq}(x)$ such that
if $\mathrm{AtSeq}(x)$ holds, $\mathrm{Disq}(x)$ returns the G\"{o}del-number of the disquotation of $x$; otherwise, $\mathrm{Disq}(x)$ returns $x$.

\begin{lemma}[Formalised Disquotation]
Define a formula $F(u)$ to be the following:

\begin{multline*}
\mrm{I}(u) \land \exists \alpha  \bigl([\mathrm{AtSeq}((u)_2) \to \alpha= (u)_1=0 \lor \alpha < (u)_1] \land [\neg \mathrm{AtSeq}((u)_2) \to \alpha = (u)_1] \\
\land \mrm{I}((u)_0; \alpha; \mathrm{Disq}((u)_2) )\bigr)
\end{multline*}

Then, $\mathsf{ID}^{*}_{1} \vdash \forall u (\mrm{I}(u) \to F(u))$.
\end{lemma}

\begin{proof}
Since $F(u)$ is $\mrm{I}$-positive, it is enough to prove $\forall u (A^I(u,F) \to F(u))$, so taking any $u$ such that $A^I(u,F)$, we show $F(u)$. The cases are divided by which disjunct of $A^I(u,F)$ holds.
\begin{description}
\item[$(\T \mhyph \mrm{Intro})$] Assume that $A^I(u,F)$ is true by $(\T \mhyph \mrm{Intro})$.
Then, $(u)_2$ is of the form $\varGamma \cup \{ \subdot{\T}t \}$ for some $t^{\circ} \in \mathrm{Sent}$ and $\varGamma \in \mathrm{Seq}$. Since $(u)_1 > 0$ by the condition of $(\T \mhyph \mrm{Intro})$, the induction hypothesis is that $F((u)_0;\alpha;\{t^{\circ}\})$ for some $\alpha < (u)_1$.
Since $\mrm{I}(u)$ is obvious from the induction hypothesis,
we show the second conjunct of $F(u)$.
If $\neg \mathrm{AtSeq}((u)_2)$, then we have to derive $\mrm{I}(u)$, which is already obtained.
If $\mathrm{AtSeq}((u)_2)$, then $\mrm{I}((u)_0; \alpha; \mathrm{Disq}((u)_2) )$ is obtained by the induction hypothesis $F((u)_0;\alpha;\{t^{\circ}\})$ and Formalised Weakening. Thus, in both cases, we have $F(u)$.

\item[$(\mrm{Cons})$] Assume that $A^I(u,F)$ is true by $(\mrm{Cons})$.
Then, $(u)_0=1$ and there exist $j \leq 1$ and a numeral $n$ such that 
the induction hypotheses are $F(j;(u)_1;(u)_2 \cup \{ \subdot{\T}n \})$ and $F(j;(u)_1;(u)_2 \cup \{ \subdot{\T}(\subdot{\neg}n) \})$. Since $\mrm{I}(u)$ is clear from the induction hypotheses, we concentrate on the second conjunct of $F(u)$. Now, we can assume $\mathrm{AtSeq}((u)_2)$, then we also have $\mathrm{AtSeq}((u)_2 \cup \{  \subdot{\T}n \})$ and $\mathrm{AtSeq}((u)_2 \cup \{ \subdot{\T}(\subdot{\neg}n) \})$.

If $n^{\circ} \notin \mathrm{Sent}$, then it follows that $\mathrm{Disq}((u)_2)= \mathrm{Disq}((u)_2 \cup \{  \subdot{\T}n \})$.
Therefore, we obtain $\mrm{I}((u)_0; \alpha; \mathrm{Disq}((u)_2))$ from the induction hypothesis $\mrm{I}(j; \alpha; \mathrm{Disq}((u)_2) \cup \{ \subdot{\T}t \})$, where $\alpha$ is an ordinal such that $\alpha=(u)_1=0$ or $\alpha<(u)_1$ holds.
If $n^{\circ} \in \mathrm{Sent}$, then we have $\mathrm{Disq}((u)_2 \cup \{  \subdot{\T}n \}) = \mathrm{Disq}((u)_2) \cup \{ n^{\circ} \}$ and $\mathrm{Disq}((u)_2 \cup \{  \subdot{\T}(\subdot{\neg}n) \}) = \mathrm{Disq}((u)_2) \cup \{ \subdot{\neg} n^{\circ} \}$, respectively. 
Thus, Formalised Cut-admissibility implies $\mrm{I}((u)_0; \alpha; \mathrm{Disq}((u)_2))$ from the induction hypotheses $\mrm{I}(j; \alpha; \mathrm{Disq}((u)_2) \cup \{ n^{\circ} \})$ and $\mrm{I}(j; \alpha; \mathrm{Disq}((u)_2) \cup \{ \subdot{\neg} (n^{\circ}) \})$, where $\alpha$ is an ordinal such that $\alpha=(u)_1=0$ or $\alpha<(u)_1$ holds.
Thus, in both cases, we have $F(u)$.
\end{description}
The other cases are proved in a similar way.
\end{proof}

\begin{corollary}[Formalised Elimination of $(\T \mhyph \mrm{Intro})$]
Fix any ordinal $\alpha < \varepsilon_0$. 
Then, we have the following:
\[
\mrm{ID}^{*}_{1} \vdash \forall \alpha_0 < \alpha \forall s,t \in \mathrm{CT} (\mrm{I}(i;\alpha_0;\{ s \subdot{=} t \}) \to \mrm{I}(i;0; \{s \subdot{=} t \})).
\]
\end{corollary}

\begin{proof}
The proof is by formal transfinite induction up to $\alpha$, which is available in $\mrm{ID}^*_1$.
\end{proof}

\begin{lemma}[Formalised Elimination of $(\mrm{Cons})$ and $(\mrm{Norm})$]
\[
\mrm{ID}^{*}_{1} \vdash \forall \varGamma \in \mathrm{AtSeq} (\mrm{I}(i;0;\varGamma) \to \mrm{I}(0;0;\mathrm{Disq}(\varGamma)).
\]
\end{lemma}

\begin{proof}
Letting $F(u):\equiv \mrm{I}(0;(u)_1;\mathrm{Disq}((u)_2))$, we show $\forall u (\mrm{I}(u) \to F(u))$.
Since $F(u)$ is $\mrm{I}$-positive, it suffices to prove $\forall u (A^I(u,F) \to F(u))$.
The remaining part of the proof is similar to for Formalised Disquotation.
\end{proof}

\begin{corollary}[Formalised Consistency of $\mrm{I}$] 
Fix any ordinal $\alpha < \varepsilon_0$. 
Then,
\[
\mrm{ID}^{*}_{1} \vdash \forall \alpha_0 < \alpha \big( \mrm{I}(i;\alpha; \{ s\subdot{=}t \}) \to s^{\circ}=t^{\circ} \big).
\]
\end{corollary}

Finally, we obtain the formalised soundness theorem.
For each natural number $n$, $\mathrm{Pos}_n(x)$ expresses that $x$ is a $\T$-positive sequent whose sentences are of at most logical complexity $\leq n$.
For each $n$, the \emph{partial} satisfaction predicate $\models_{n}^{y}(x)$
means that the sentence $x$ has at most logical complexity $\leq n$ and $x$ is satisfied at the level $y < \varepsilon_0$.

In particular, $\models^{y}_n(x)$ is defined such that the following are satisfied, provably even in $\mrm{PA}$:
\begin{itemize}

\item $\models^{y }_n(\ulcorner s = t \urcorner) \ \leftrightarrow \ s = t$.

\item $\models^{y }_n(\ulcorner s \neq t \urcorner) \ \leftrightarrow \ s \neq t$.
\item $\models^{y }_n(\ulcorner \T (t) \urcorner) 
 \ \leftrightarrow \ \mrm{I}(1;y;\{ t \})$.
 
\item $\models^{y }_n (\ulcorner A \land B \urcorner) \ \leftrightarrow \ \models^{y}_{n-1} (\ulcorner A \urcorner) \land \models^{y}_{n-1} (\ulcorner B \urcorner)$, where $A \land B \in \lt$ and $\mrm{co}(A \land B) \leq n$.

\item $\models^{y }_n (\ulcorner A \lor B \urcorner) \ \leftrightarrow \ \models^{y}_{n-1} (\ulcorner A \urcorner) \lor \models^{y}_{n-1} (\ulcorner B \urcorner)$, where $A \lor B \in \lt$ and $\mrm{co}(A \land B) \leq n$.

\item $\models^{y }_n (\ulcorner \forall x A(x) \urcorner) \ \leftrightarrow \ \forall v \big(\models^{y}_{n-1} (\ulcorner A(\dot{v}) \urcorner) \big)$, where $\forall x A(x) \in \lt$ and $\mrm{co}(\forall x A(x)) \leq n$.

\item $\models^{y }_n (\ulcorner \exists x A(x) \urcorner) \ \leftrightarrow \ \exists v \big(\models^{y}_{n-1} (\ulcorner A(\dot{v}) \urcorner) \big)$, where $\exists x A(x) \in \lt$ and $\mrm{co}(\forall x A(x)) \leq n$.

\item $\models^{y }_n (\ulcorner A \urcorner) \ \leftrightarrow \ 0 = 1$, if 
$\mrm{co}(A) > n$ or $A$ is not $\T$-positive.
\end{itemize}

Then, we can easily expand the definition to sequents $x \in \mrm{Seq}$:
\[
\models^{y}_n (x) : \leftrightarrow \ x \in \mrm{Seq} \land \exists z \in x \big(\models^{y}_n(z)\big).
\]

When $x$ is the singleton of an $\lnat$-sentence $A$, we can prove, by meta-induction on $A$, that this predicate implies $A$ itself for every $n \geq \mrm{co} (A)$:
\[
\mrm{ID}^{*}_{1} \vdash \ \models^{y}_n(\{ \ulcorner A \urcorner \}) \to A.
\]


The arithmetical predicate $\mathrm{Bew}_{\mrm{VFM}^{\infty}}(x,y,z)$ stands for the sequent $z$ having a \emph{recursive} derivation with the length $x$ and the cut-rank $y$ in $\mrm{VFM}^{\infty}$.\footnote{For further explanation of such a predicate, see, e.g. \cite[5.2.2]{schwichtenberg_1977}. 
See also the next section, where we give a detailed explanation of an infinite derivability predicate for $\mrm{VFW}$, using ramified truth predicates.}

\begin{lemma}[Formalised Persistency of $\models$]\label{lem:form-persist-I}
Fix any ordinal number $\alpha < \varepsilon_0$ and any natural number $n$.
Then, for a sufficiently large $m \geq n$, we have the following:
\[
\mathsf{ID}^{*}_{1} \vdash \forall \alpha_1 < \alpha \forall \alpha_0 < \alpha_1 \forall x \in \mathrm{Pos}_n (\models^{\alpha_0}_m(x) \to \models^{\alpha_1}_m(x)).
\]
\end{lemma}

\begin{lemma}[Formalised Soundness]\label{form soundness}
Fix any ordinal number $\alpha < \varepsilon_0$ and any natural number $n$. Then, for a sufficiently large $m \geq n$, we have the following:
\[
\mathsf{ID}^{*}_{1} \vdash \forall \alpha_0 < \alpha \forall x \in \mathrm{Pos}_n (\mathrm{Bew}_{\mrm{VFM}^{\infty}}(\alpha_0,0, x) \to \models^{\alpha_0}_{m}(x)).
\]
\end{lemma}

\begin{proof}
The proof is by formal transfinite induction up to $\alpha$.
\end{proof}

From the proof of the embedding lemma, $\mrm{PA}$ is clearly enough to formalise this fact. Thus, we have:
\begin{lemma}[Formalised Embedding lemma]\label{lemma:formalized_embedding}
Assume that $\mrm{VFM} \vdash A$ for some $\lt$-sentence $A$.
Then, we have $\mrm{PA} \vdash \mathrm{Bew}_{\mrm{VFM}^{\infty}}(\alpha,0, \{ \ulcorner A \urcorner \})$ for some $\alpha < \varepsilon_0$.
\end{lemma}

\begin{thm}\label{thm:upper-bound_VFM}
Let $A$ be any $\lnat$-sentence.
If $\mrm{VFM} \vdash A$, then $\mrm{ID}^{*}_{1} \vdash A$. Therefore, $\vert\mrm{VFM}\vert \leqq \vert \mrm{PA} + \mrm{TI}(< \varphi_{\varepsilon_0} 0)\vert $.
\end{thm}

\begin{proof}
Assume that $\mrm{VFM} \vdash A$. Then, by the formalised embedding lemma (Lemma \ref{lemma:formalized_embedding}), we have that $\mrm{PA} \vdash \mathrm{Bew}_{\mrm{VFM}^{\infty}}(\alpha,k, \{ \ulcorner A \urcorner \})$ for some $k$ and $\alpha < \varepsilon_0$.
Thus, by formalising Cut-ellimination for $\mrm{VFM}^{\infty}$, we also have $\mrm{PA} \vdash \mathrm{Bew}_{\mrm{VFM}^{\infty}}(\omega_k(\alpha),0, \{ \ulcorner A \urcorner \})$.
Therefore, by formalised soundness, we obtain $\mrm{ID}^{*}_{1} \vdash\models^{\omega_k(\alpha)}_{m} (\{ \ulcorner A \urcorner \})$ for a sufficiently large $m \geq \mrm{co}(A)$.
Since we can prove $\mrm{ID}^{*}_{1} \vdash \models^{\omega_k(\alpha)}_{m} (\{ \ulcorner A \urcorner \}) \to A$ by an induction on the complexity of $A$,
it follows that $\mrm{ID}^{*}_{1} \vdash A$.
\end{proof}


\begin{proof}[Proof of Theorem \ref{upper bound}]
Follows from Theorem \ref{thm:upper-bound_VFM}.
\end{proof}

By the results of the last two sections:

\begin{thm}
$\vert\mrm{VFM}\vert \equiv \vert \mrm{PA} + \mrm{TI}(< \varphi_{\varepsilon_0} 0)\vert \equiv \vert \mrm{KF}\vert \equiv \vert \mrm{RT}_{<\varepsilon_0}\vert $
\end{thm}

\section{The Theories VFW$^-$ and VFW}\label{section:VFW}

As we have seen in the previous sections, $\mrm{VFM}^{(-)}$ is much weaker than $\mrm{VF}^{(-)}$ 
and it would be natural to speculate that the axiom of \emph{internal completeness}, i.e., the left-to-right direction $(\mrm{VF}7^{\to})$ of $\mrm{VF}7$, is the cause of the weakness:
\begin{description}
\item[$(\mrm{VF}7^{\to})$] $\T(\ulcorner \neg \T(\dot x) \rightarrow  \T(\subdot \neg \dot x)\urcorner)$
\item[$(\mrm{VF}7^{\leftarrow})$] $\T(\ulcorner \neg \T(\dot x) \leftarrow  \T(\subdot \neg \dot x)\urcorner)$
\end{description}
Indeed, by Friedman and Sheard's result, we can easily observe that even a small fragment of $\mrm{VF}$ is incompatible with $\mrm{VF}7^{\to}$.
Let $\T \mhyph \mrm{Elim}$ be the following rule:
\begin{center} \
\infer[]{A}{\T \ulcorner A \urcorner}, for any $\lt$-sentence $A$.
\end{center}

\begin{obse}[{cf.~\cite[Section 4]{friedman_sheard_1987}}]
$\mrm{V}2$, $\mrm{V}4$, $\mrm{V}6$,  $\mrm{VF}7^{\to}$, and $\T \mhyph \mrm{Elim}$ combinedly imply a contradiction. 
\end{obse}

Thus, $(\mrm{VF}7^{\to})$ and $\T \mhyph \mrm{Elim}$ are in the relationship of trade-off over $\mrm{VFM}^{(-)}$.
In this section, we study the strength of $\T \mhyph \mrm{Elim}$ to look more closely at the difference between the two.

The theory $\mrm{VFW}$ is obtained from $\mrm{VFM}$ by adding $ (\T \mhyph \mrm{Elim})$ and by replacing $(\mrm{VF}7)$ with $(\mrm{VF}7^{\leftarrow})$.
The theory $\mrm{VFW}^-$ is similarly obtained from $\mrm{VFM}^-$,
then $\mrm{VFW}^-$ is of course a subtheory of $\mrm{VFW}$.
Since $\mrm{VFW}$, and hence $\mrm{VFW}^-$, is clearly a subtheory of $\mrm{VF}$, it is consistent.
In particular, we determine the exact proof-theoretic strength: 

\begin{thm}\label{thm:VFW}
$\vert \mrm{VFW}^- \vert \equiv \vert \mrm{VFW} \vert \equiv \vert \mrm{PA} + \mrm{TI}(< \varphi_{\varphi_2 0} 0)\vert \equiv \vert \mrm{RT}_{< \varphi_20} \vert$.
\end{thm}


Before exhibiting the proof, we give an outline.
First, the lower bound is immediate, because Leigh and Rathjen \cite[Theorem~2.41]{rathjen_leigh_2010} prove that (a subtheory of) $\mrm{VFW}^-$ derives the schema $\mrm{TI}_{\lt}(< \varphi_20)$.
Thus, by applying the proof of Theorem~\ref{thm:lower_bound_VFM^-}, we can define the system of ramified truth $\mrm{RT}_{< \varphi_20}$ in $\mrm{VFW}^-$.
\footnote{In particular, note that $\mrm{VF}7$ is not used in the proof of Lemma~\ref{properties xi star}.}

To determine the upper bound of $\mrm{VFW}$, we need to modify the upper-bound proof for $\mrm{VFM}$, in which we used two derivation systems $\mrm{VFM}^{\infty}$ and $I(x;y;z;w)$.
Here, $\mrm{VFM}^{\infty}$ is used to embed $\mrm{VFM}$, and $I(x;y;z;w)$ is for the interpretation of the truth predicate.
For the proof-theoretic purpose, we considered only \emph{recursive} derivations of 
$\mrm{VFM}^{\infty}$, whereas $I(x;y;z;w)$ need not be recursive. 
Now, if one were to give a similar proof for $\mrm{VFW}$, it would proceed as follows:
\begin{enumerate}
\item We define a derivation system $\mrm{VFW}^{\infty}$ for which it is proved that $\mrm{VFW}$ is embeddable.
In particular, we want to show that $\mathsf{VFW}^{\infty}$ is closed under $\T \mhyph \mrm{Elim}$.

\item Similar to $I(x;y;z;w)$, we define a derivation system $I'$ such that if $\mrm{VFW}^{\infty} \vdash \T \ulcorner A \urcorner$, then $I' \vdash A$ holds.

\item We also prove that if $I' \vdash A$, then $\mrm{VFW}^{\infty} \vdash A$.
Therefore, $\mrm{VFW}^{\infty}$ is closed under $\T \mhyph \mrm{Elim}$, as required.
\end{enumerate}

The last step entails that $I' \subseteq \mrm{VFW}^{\infty}$, but since $I'$ has the (non-recursive) $\omega$-rule, $\mrm{VFW}^{\infty}$ must also be closed under it. Thus  $\mrm{VFW}^{\infty}$ cannot be restricted to recursive derivations.
To solve this problem, we, in $\mrm{RT}_{< \varphi_20}$, define both systems $\mrm{VFW}^{\infty}$ and $I'$ in which derivation lengths are restricted to less than $\varphi_20$.
It should be noted here that we now have $\mrm{TI}_{\mathcal{L}_{< \varphi_20}}(< \varphi_20)$ in $\mrm{RT}_{< \varphi_20}$,
which is enough to directly  formalise $\mrm{VFW}^{\infty}$, $I'$, and their properties, such as the soundness theorem for $I'$. 
Therefore, we do not need to rely on $\mrm{ID}^*_1$, unlike for the upper-bound proof of $\mrm{VFM}$.


Keeping the above motivation in mind, we now define the systems $\mrm{VFW}^{\infty}$ and $I'(x;y;z;w)$. 
First, $I'(i;\alpha;\beta;\varGamma) \subseteq \{0,1\} \times \varphi_20 \times \varphi_20 \times \mrm{Seq}$ is obtained from $I(i;\alpha;\beta;\varGamma)$ by simply removing the axiom $(\mrm{Comp})$ of Definition~\ref{dfn:I}.
The meanings of $i$, $\alpha, \beta$, and $\varGamma$ are exactly the same as those for $I$. To summarise, $I'$ is defined as follows:

\begin{dfn}[Definition of $I'$]\label{defn:I'}
The set $I' \subseteq \{0,1\} \times \varphi_20 \times \varphi_20 \times \mrm{Seq}$ is defined to be the least fixed-point which is closed under the clauses below.
We write $I'(i;\alpha;\beta;\varGamma)$ instead of $\langle i,\alpha,\beta, \varGamma \rangle \in I'$.

Let $i \in \{0,1\}$, $\alpha_0, \alpha, \beta_0, \beta \in \mrm{On}$, and $\varGamma \in Seq$.
Furthermore, we assume $\alpha_0 < \alpha$, $\beta_0 < \beta$, and $\alpha \leq \beta$.
\begin{description}
\item[$(\mathrm{Ax}.1)$] $I'(i; \alpha; \beta; \varGamma, s=t)$ holds, if $s=t$ is true.

\item[$(\mathrm{Ax}.2)$] $I'(i; \alpha; \beta; \varGamma, s \neq t)$ holds, if $s\neq t$ is true.

\item[$(\mathrm{Ax}.3)$] $I'(i; \alpha; \beta; \varGamma, \T(s), \neg \T(t \simeq s))$ holds.

\item[$(\lor)$] If $I'(i;\alpha; \beta_0;\varGamma, A_0,A_1)$, then $I'(i;\alpha; \beta;\varGamma, A_0 \lor A_1)$.

\item[$(\land)$] If $I'(i;\alpha; \beta_0; \varGamma, A_0)$ and $I'(i;\alpha; \beta_0; \varGamma, A_1)$, then $I'(i;\alpha; \beta;\varGamma, A_0 \land A_1)$.

\item[$(\exists)$] If $I'(i;\alpha; \beta_0; \varGamma, A(n))$ for some $n$, then $I'(i;\alpha; \beta; \varGamma, \exists x A(x))$.

\item[$(\forall)$] If $I'(i;\alpha;\beta_0;\varGamma, A(n))$ for all $n$, then $I'(i;\alpha;\beta;\varGamma, \forall x A(x))$.

\item[$(\T \mhyph \mrm{Intro})$] If $I'(i;\alpha_0;\beta_0;A)$, then 
$I'(i;\alpha;\beta;\varGamma, \T (t \simeq \ulcorner A \urcorner))$.


\item[$(\mrm{Cons})$] If $I'(i;\alpha;\beta_0;\varGamma, \T(n))$ and $I'(i;\alpha;\beta_0;\varGamma, \T(\subdot{\neg}n))$, then $I'(1;\alpha;\beta;\varGamma)$.

\item[$(\mrm{Norm})$] If $I'(i;\alpha;\beta_0;\varGamma, \T(n))$ and $\neg \mathrm{Sent}(n)$ holds, then $I'(1;\alpha;\beta;\varGamma)$.

\end{description}
\end{dfn}

Second, the system $\mrm{VFW}^{\infty}$ is obtained from $\mrm{VFM}^{\infty}$ by removing the axiom $(\T_{\mrm{Comp}})$, and instead adding the following axiom $(\mathrm{Ax}_{I'})$ and rule $(\mrm{Norm})$:
\begin{description}
\item[$(\mathrm{Ax}_{I'})$] $\sststile{k}{\beta}\varGamma, \T(t \simeq \ulcorner A \urcorner)$ holds, if $I'(i;\alpha;\beta;A)$.
\end{description}

\begin{center} \
\infer[(\mrm{Norm})]{\sststile{k}{\alpha} \varGamma}{\sststile{k}{\alpha_0} \varGamma, \T(n) \text{ for $\mathbb{N} \not\models \mathrm{Sent}(n)$}}.
\end{center}

The new axiom and rule are used to prove Lemma~\ref{lem:I'_subset_VFW-infty} below. We remark that Lemma~\ref{lem:properties_VFM_inf} (Substitution, Weakening, and Cut elimination) holds for $\mrm{VFW}^{\infty}$. By contrast, unlike with Lemma~\ref{lem:embedding_VFM},
the proof of the Embedding Lemma for $\mrm{VFW}$ is more complicated due to $\T \mhyph \mrm{Elim}$.
In order to prove the admissibility of $\T \mhyph \mrm{Elim}$ in $\mrm{VFW}^{\infty}$, the Soundness Lemma must be established.

We begin by establishing the following, which roughly means $I' \subseteq \mrm{VFW}^{\infty}$:
\begin{lemma}\label{lem:I'_subset_VFW-infty}
If $I'(i;\alpha;\beta;\varGamma)$, then $\mrm{VFW}^{\infty}\sststile{0}{\beta}\varGamma$. 
\end{lemma}

\begin{proof}
By induction on $\beta$.
We divide the cases by the last rule of the derivation of $I'$,
but it is sufficient to observe the case of the rule $(\T \mhyph \mrm{Intro})$, because the other rules are shared by $I'$ and $\mrm{VFW}^{\infty}$.
So, letting $\varGamma = \varGamma' \cup \{ \T(t) \}$ for some $t$ with $t^{\mathbb{N}} = \ulcorner A \urcorner$, 
we assume that $I'(i;\alpha;\beta;\varGamma)$ is derived from $I'(i;\alpha;\beta_0;A)$ for some $\beta_0 < \beta$.
Then, by the axiom $(\mathrm{Ax}_{I'})$ and Weakening in $\mrm{VFW}^{\infty}$, we obtain $\sststile{0}{\beta}\varGamma$, as desired.
\end{proof}

Next, the definition of the satisfaction relation $\models^{\alpha} A$ is the same as in Section~\ref{sec:truth-as-provability_VFM}, except for the clause for $\T$.
The only change is to interpret the truth predicate by $I'$ instead of $I$:
\begin{itemize}
\item $\models^{\alpha} \T(t)$ $:\Leftrightarrow$ the value of $t$ is the G\"{o}del-number of some $\lt$-sentence $A$ and the $4$-ary relation $I'(1;\alpha; \varepsilon_{\alpha}; A)$ holds.
\end{itemize}

For a $\T$-positive sequent $\varGamma$, let $\models^{\alpha} \varGamma$ $:\Leftrightarrow$ $\models^{\alpha} A$ for some $A \in \varGamma$.


Our aim is then to prove the soundness of $\mrm{VFW}^{\infty}$.
Since $I'$ is just a subsystem of $I$, 
we get the same results for $I'$ as for $I$, so we list them without proof:

\begin{lemma}[Substitution for $I'$]
If $I'(i;\alpha;\beta;\varGamma,A(s))$ and $s=t$ is true, then $I'(i;\alpha;\beta;\varGamma,A(t))$.
\end{lemma}

\begin{lemma}[Weakening for $I'$]
Assume $0 \leq i \leq j \leq 1$
; $\alpha_0 \leq \alpha$; $\beta_0 \leq \beta$; $\alpha \leq \beta$;
and $\varGamma_0 \subseteq \varGamma$.
If $I'(i;\alpha_0;\beta_0;\varGamma_0)$, then $I'(j;\alpha;\beta;\varGamma)$.
\end{lemma}

\begin{lemma}[Cut-admissibility for $I'$]\label{lem:cut-adm_I'}
If $I'(i;\alpha;\beta;\varGamma, A)$ and $I'(i;\alpha;\gamma;\varDelta,\neg A)$, then it holds that $I'(i;\alpha;\omega_{\mathrm{co}(A)}(\beta \# \gamma);\varGamma,\varDelta)$, 
where $\mathrm{co}(A)$ is the logical complexity of $A$.
\end{lemma}


\begin{lemma}[Diquotation for $I'$]\label{lem:disq_I'}
Let $\varGamma$ be an atomic sequent and assume $I'(i;\alpha;\beta;\varGamma)$. Then, $I'(i;\alpha_0;\omega_n(\beta);\mrm{Disq}(\varGamma))$ holds for some $n \in \mathbb{N}$ and $\alpha_0 < \alpha$.
In particular, when $\varGamma$ contains only equations, we obtain $I'(i;0;\varepsilon_{\beta};\varGamma)$.
\end{lemma}



\begin{lemma}[Elimination of $(\mrm{Cons})$ and $(\mrm{Norm})$ in $I'$]
Assume $I'(i;0;\beta;\varGamma)$ for an atomic sequent $\varGamma$. Then, $I'(0;0;\omega_{n}(\beta);\mrm{Disq}(\varGamma))$ holds for some $n \in \mathbb{N}$.
Therefore, if $\varGamma$ contains only equations, then we obtain $I'(0;0;\omega_n(\beta);\varGamma)$.
\end{lemma}


\begin{corollary}[Consistency of $I'$]\label{cor:consistency_of_I'}
No false equation $s=t$ is derivable in $I'$, that is, if $I'(i;\alpha;\beta;s=t)$, then $s=t$ is true. 
\end{corollary}

Persistency is established, just like in Lemma~\ref{lem:persist-I}:

\begin{lemma}[Persistency of $\models$]
Let $\varGamma$ be a $\T$-positive sequent. 
If $\models^{\alpha_0} \varGamma$, then $\models^{\alpha} \varGamma$ for any $\alpha > \alpha_0$. 
\end{lemma}

Using the above lemmata, we can prove the Soundness Lemma for $I'$  in a similar way as for $I$.

\begin{lemma}[Soundness of $\mrm{VFW}^{\infty}$]\label{lem:soundness_VFW-infty}
Let $\varGamma$ be any $\T$-positive sequent and assume $\alpha < \varphi_20$.
If $\sststile{0}{\alpha}\varGamma$, then $\models^{\alpha} \varGamma$. 
\end{lemma}

\begin{proof}[Proof of Lemma~\ref{lem:soundness_VFW-infty}]
The proof is by induction on $\alpha$ and is almost the same as for Theorem~\ref{thm:soundness}.
Thus, it is sufficient to consider the new cases, $(\mrm{Ax}_{I'})$ and $(\mrm{Norm})$.
\begin{description}
\item[$(\mrm{Ax}_{I'})$] We assume that $\sststile{0}{\alpha}\varGamma, \T(t \simeq \ulcorner A \urcorner)$ holds by $(\mrm{Ax}_{I'})$. Then, by the condition of $(\mrm{Ax}_{I'})$, we have $I'(i;\gamma;\alpha;A)$ for some $i \leq 1$ and $\gamma \leq \alpha$.
Therefore, we obtain by Weakening for $I'$ that $I'(1;\alpha;\varepsilon_{\alpha};A)$, as desired.


\item[$(\mrm{Norm})$] The case for $(\mrm{Norm})$ is obvious, for $\models^{\alpha_0} \T(n)$ does not hold if $n$ denotes no sentence.

\end{description}


\end{proof}

Finally, we observe that
Lemma~\ref{lem:soundness_VFW-infty} yields the admissibility of $\T \mhyph \mrm{Elim}$ and hence the Embedding Lemma for $\mrm{VFW}$:

\begin{corollary}[Admissibility of $\T \mhyph \mrm{Elim}$]\label{cor:adm_T-Elim}
Assume $\alpha < \varphi_20$.
If $\mrm{VFW}^{\infty}\sststile{0}{\alpha} \T \ulcorner A \urcorner$, then $\mrm{VFW}^{\infty}\sststile{0}{\varepsilon_{\alpha}} A$.
\end{corollary}

\begin{proof}
Assume $\mrm{VFW}^{\infty}\sststile{0}{\alpha} \T \ulcorner A \urcorner$.
By Lemma~\ref{lem:soundness_VFW-infty}, we have $\models^{\alpha} \T \ulcorner A \urcorner$, thus $I'(1;\alpha;\varepsilon_{\alpha};A)$ holds,
which, by Lemma~\ref{lem:I'_subset_VFW-infty}, implies
$\mrm{VFW}^{\infty}\sststile{0}{\varepsilon_{\alpha}} A$. 
\end{proof}

\begin{corollary}[Embedding]\label{cor:embed_VFW}
If $\mrm{VFW} \vdash A$, then $\mrm{VFW}^{\infty} \sststile{0}{\alpha} A$ for some $\alpha < \varphi_20$.
\end{corollary}

\begin{proof}
Similar to the proof of Lemma~\ref{lem:embedding_VFM}, the claim is proved by induction on the length of the derivation of $A$ in $\mrm{VFW}$.
For the case of $\T \mhyph \mrm{Elim}$, we use Corollary~\ref{cor:adm_T-Elim}.
Note that in the derivation of $A$ in $\mrm{VFM}$, the number of applications of $\T \mhyph \mrm{Elim}$ is at most finite, so the derivation length $\alpha$ of $A$ in $\mrm{VFW}^{\infty}$ can be kept below $\varphi_20$.
\end{proof}

As the last step to obtain the upper bound, we need to formalise the above arguments in $\mrm{RT}_{<\varphi_20}$.
First, we define a predicate $\mrm{Bew}_{\mrm{VFW}^{\infty}}(x,y,z)$, which means that a sequent $z$ is derivable in $\mrm{VFW}^{\infty}$ with the height $x < \varphi_20$ and with the cut rank $y \in \mathbb{N}$.
Recall that we need to consider non-recursive derivations in $\mrm{VFW}^{\infty}$, but the language of first-order arithmetic is not sufficient for expressing that.
Thus, the predicate must be defined as a formula of $\mc{L}_{<\varphi_20}$.
Intuitively, $\mrm{Bew}_{\mrm{VFW}^{\infty}}(\alpha,y,z)$ is defined by transfinite recursion on $\alpha$, according to the definition of $\mrm{VFW}^{\infty}$.
The base case $\alpha = 0$ is definable as an arithmetical formula:
\[
\mrm{Bew}_{\mrm{VFW}^{\infty}}(0,y,z) \ \text{iff} \ z \ \text{is (the code of) an instance of} \ (\mrm{Ax}.1), (\mrm{Ax}.2), (\T_{\mrm{Cons}}), (\T_{\mrm{Norm}}), \ \text{or} \ (\mrm{Ax}_{I'}).
\]
For $0 < \alpha < \varphi_20$, $\mrm{Bew}_{\mrm{VFW}^{\infty}}(\alpha,y,z)$ is defined according to the last rule of the derivation.
For example, the rule $(\forall)$ is expressed such that
$\mrm{Bew}_{\mrm{VFW}^{\infty}}(\alpha,y,z)$ is implied by the following formula:
\[
\exists \varGamma' \in \mrm{Seq} \exists v, x \big( z = \varGamma' \cup \{ \subdot{\forall}vx \} \land \forall n \exists \alpha_n < \alpha \T_{\alpha}(\ulcorner \mrm{Bew}_{\mrm{VFW}^{\infty}}(\dot{\alpha_n},\dot{y},\dot{\varGamma'} \cup \{ \dot{x}(\dot{n}/\dot{v}) \}) \urcorner) \big).
\]

Here, the predicate $\T_{\alpha}$ is used for expressing the premises of $(\forall)$, i.e., an infinite conjunction of $\mrm{Bew}_{\mrm{VFW}^{\infty}}(\alpha_0,y,\varGamma' \cup \{ x(0/v) \}), \mrm{Bew}_{\mrm{VFW}^{\infty}}(\alpha_1,y,\varGamma' \cup \{ x(1/v) \}), \dots$.

The other cases are similar.
Note that for each $\alpha < \varphi_20$, the predicate $\mrm{Bew}_{\mrm{VFW}^{\infty}}(\alpha,y,z)$ so defined is obviously a formula of $\mc{L}_{\alpha}$. 

To define the above construction formally,
we first define the code $ \ulcorner \mrm{Bew}_{\mrm{VFW}^{\infty}}(\dot{x},\dot{y},\dot{z}) \urcorner$ by using the primitive recursion theorem.
Then, taking a particular ordinal $\alpha < \varphi_20$, the formula $\T_{\alpha} \ulcorner \mrm{Bew}_{\mrm{VFW}^{\infty}}(\dot{x},\dot{y},\dot{z}) \urcorner$ can indeed play the role of the predicate $\mrm{Bew}_{\mrm{VFW}^{\infty}}(x,k,\varGamma)$ for each $x < \alpha$.
Hence, precisely speaking, the predicate $\mrm{Bew}$ is defined relative to a particular ordinal $\alpha < \varphi_20$.
To make this explicit, we write $\mrm{Bew}_{\mrm{VFW}^{\infty}}^{\alpha}(x,y,z)$.
Then, we can verify that it satisfies the properties of $\mrm{VFW}^{\infty}$ up to the length $< \alpha$. 
For example, the rule $(\forall)$ is now expressed as follows for each $\alpha < \varphi_20$:
\[ \mrm{RT}_{< \varphi_20} \vdash 
\forall \beta < \alpha \big( [\forall n \exists \beta_n < \beta \mrm{Bew}_{\mrm{VFW}^{\infty}}^{\alpha}(\beta_n, k, \varGamma \cup \{ x(n/v)\})] \to \mrm{Bew}_{\mrm{VFW}^{\infty}}^{\alpha}(\beta, k, \varGamma \cup \{ \subdot{\forall}vx \}) \big). 
\]

Secondly, derivability in $I'$ up to the length $< \varepsilon_{\alpha}$ is similarly defined as a 4-ary predicate $\mrm{I}'_{\alpha}(x;y;z;w)$ for each $\alpha < \varphi_20$.
Using the predicate $\mrm{I}'_{\alpha}$, the binary satisfaction predicate $\models^{y < \alpha}_n(x)$ is defined by meta-induction on $n$, in the same way as for $\mrm{I}$. 
Its intuitive meaning is that the sentence $x$ has the logical complexity $\leq n$ and $x$ is satisfied at the level $y < \alpha < \varphi_20$.
In particular, $\models^{y < \alpha}_n(x)$ is defined such that the following are satisfied, provably even in $\mrm{PA}$:
\begin{itemize}

\item $\models^{y < \alpha}_n(\ulcorner s = t \urcorner) \ \leftrightarrow \ s = t$.

\item $\models^{y < \alpha}_n(\ulcorner s \neq t \urcorner) \ \leftrightarrow \ s \neq t$.
\item $\models^{y < \alpha}_n(\ulcorner \T (t) \urcorner) 
 \ \leftrightarrow \ \mrm{I}'_{\alpha}(1;y;\varepsilon_{y};\{ t \})$.
 
\item $\models^{y < \alpha}_n (\ulcorner A \land B \urcorner) \ \leftrightarrow \ \models^{y < \alpha}_{n-1} (\ulcorner A \urcorner) \land \models^{y < \alpha}_{n-1} (\ulcorner B \urcorner)$, where $A \land B \in \lt$ and $\mrm{co}(A \land B) \leq n$.

\item $\models^{y < \alpha}_n (\ulcorner A \lor B \urcorner) \ \leftrightarrow \ \models^{y < \alpha}_{n-1} (\ulcorner A \urcorner) \lor \models^{y < \alpha}_{n-1} (\ulcorner B \urcorner)$, where $A \lor B \in \lt$ and $\mrm{co}(A \land B) \leq n$.

\item $\models^{y < \alpha}_n (\ulcorner \forall x A(x) \urcorner) \ \leftrightarrow \ \forall v \big(\models^{y < \alpha}_{n-1} (\ulcorner A(\dot{v}) \urcorner) \big)$, where $\forall x A(x) \in \lt$ and $\mrm{co}(\forall x A(x)) \leq n$.

\item $\models^{y < \alpha}_n (\ulcorner \exists x A(x) \urcorner) \ \leftrightarrow \ \exists v \big(\models^{y < \alpha}_{n-1} (\ulcorner A(\dot{v}) \urcorner) \big)$, where $\exists x A(x) \in \lt$ and $\mrm{co}(\forall x A(x)) \leq n$.

\item $\models^{y < \alpha}_n (\ulcorner A \urcorner) \ \leftrightarrow \ 0 = 1$, if 
$\mrm{co}(A) > n$ or $A$ is not $\T$-positive.
\end{itemize}

Then, we can easily expand the definition to sequents $x \in \mrm{Seq}$:
\[
\models^{y < \alpha}_n (x) : \leftrightarrow \ x \in \mrm{Seq} \land \exists z \in x \big(\models^{y < \alpha}_n(z)\big).
\]

When $x$ is the singleton of an $\lnat$-sentence $A$, we can prove, by meta-induction on $A$, that this predicate implies $A$ itself for every $n \geq \mrm{co} (A)$ and $\alpha < \varphi_20$:
\[
\mrm{RT}_{<\varphi_20} \vdash \ \models^{y < \alpha}_n(\{ \ulcorner A \urcorner \}) \to A.
\]

With the help of these predicates, we can indeed formalise the above results in $\mrm{RT}_{< \varphi_20}$. For example, we restate Lemma~\ref{lem:I'_subset_VFW-infty}, Lemma~\ref{lem:soundness_VFW-infty}, and Corollary~\ref{cor:embed_VFW} as follows:

\begin{lemma}[cf.~Lemma~\ref{lem:I'_subset_VFW-infty}]
Fix any ordinal number $\beta < \varphi_20$. Then, we have the following:
\[
\mrm{RT}_{< \varphi_20} \vdash \forall i \leq 1 \forall \beta_0 < \beta \forall \alpha \leq \beta_0 \forall \varGamma \in \mrm{Seq}\big(\mrm{I}'_{\beta}(i;\alpha;\beta_0;\varGamma) \to \mrm{Bew}^{\beta}_{\mrm{VFW}^{\infty}}(\beta_0;0;\varGamma) \big).
\]
\end{lemma}

Recall that for each natural number $n$, the predicate $\mathrm{Pos}_n(x)$ expresses that $x$ is a $\T$-positive sequent whose sentences are of at most logical complexity $n$.
\begin{lemma}[cf.~Lemma~\ref{lem:soundness_VFW-infty}]\label{lem:formalized_soundness_VFW-infty}
Fix any ordinal number $\alpha < \varphi_20$ and any natural number $n$. Then,  for a sufficiently large $m \geq n$, we have the following:
\[
\mrm{RT}_{< \varphi_20} \vdash 
\forall \alpha_0 < \alpha \forall x \in {\rm{Pos}}_n (\mrm{Bew}^{\alpha}_{\mrm{VFW}^{\infty}}(\alpha_0,0,x) \to \models^{\alpha_0 < \alpha}_m(x)).
\]
\end{lemma}

\begin{lemma}[cf.~Corollary~\ref{cor:embed_VFW}]\label{lemma:formalized_embedding-I'}
Assume $\mrm{VFW} \vdash A$ for some $\lt$-sentence $A$.
Then, we have the following  for some $\alpha_0 < \alpha < \varphi_20$:
\[
\mrm{RT}_{<\varphi_20} \vdash \mathrm{Bew}^{\alpha}_{\mrm{VFW}^{\infty}}(\alpha_0,0, \{ \ulcorner A \urcorner \}).
\]
\end{lemma}


Therefore, we obtain the upper bound of $\mrm{VFW}$:

\begin{thm}[Upper bound of $\mrm{VFW}$]\label{thm:upper_VFW}
Assume $A \in \lnat$.
If $\mrm{VFW} \vdash A$, then $\mrm{RT}_{<\varphi_20} \vdash A$.
\end{thm}

\begin{proof}
Assume $\mathsf{VFW} \vdash A$ for $A \in \lnat$.
By Lemma~\ref{lemma:formalized_embedding-I'} and Lemma~\ref{lem:formalized_soundness_VFW-infty}, $\mrm{RT}_{< \varphi_20} \vdash \models^{\alpha_0 < \alpha}_n (\{ \ulcorner  A \urcorner \} )$ follows for some $\alpha_0 < \alpha < \varphi_20$ and some $n \geq \mrm{co}(A)$.
Since $A \in \lnat$, this implies $\mrm{RT}_{< \varphi_20} \vdash A$.
\end{proof}

\begin{proof}[Proof of Theorem \ref{thm:VFW}]
As remarked, the lower bound is obtained by applying the proof in Theorem \ref{thm:lower_bound_VFM^-} to the results obtained in \cite[Theorem~2.41]{rathjen_leigh_2010}. The upper bound is our Theorem \ref{thm:upper_VFW}.
\end{proof}


\section{The schematic extension of VFM}

In this section, we want to explore an extension of $\mrm{VFM}$ based on the idea of the schematic extension of a theory, first proposed by Feferman in \cite{feferman_1991}. Schematic extensions are interesting for two reasons. On the one hand, from a philosophical point of view, they can be seen as a formalization of the notion of \textit{implicit commitment} of a theory, or so did Feferman argue.\footnote{In recent years, the notion of implicit commitment for mathematical theories has gained a renewed attention---see e.g. \cite{nicolai_lelyk_2021, dean_2014}.} On the other hand, schematic extensions have sometimes been shown to increase the proof-theoretic strength of the theory which they extend. A paradigmatic case of this phenomenon is KF, as Feferman showed---the schematic extension of KF, $\mrm{KF}^*$, has proof-theoretic ordinal $\Gamma_0$. Meanwhile, in the case of VF (and, consequently, $\mrm{VF}^-$), the schematic extension is not associated with an increase in proof-theoretic strength---see \cite{fujimoto_2018}.

For the purposes of exploring the schematic extension of VFM, we need to work in a language $\ltp$, which extends $\mc{L}_\T$ with a schematic predicate $P$. Accordingly, $\mc{L}(P)$ will be the language of PA extended with $P$. This predicate does its job as a predicate variable for arithmetical formulae $A(P)$, allowing us then to substitute $P$ in $A$ for any formula $B$ in the language $\ltp$. Thus, the schematic extension of $\mrm{VFM}$, the theory $\mrm{VFM}^*$, is defined as follows: 

\begin{dfn}
$\mrm{VFM}^*$ consists of the axioms of $\mrm{VFM}$ (save for the axiom scheme of induction) and the following: 

\begin{itemize}
    \item The axiom 
    \begin{itemize}
        \item[P-Disq] $\forall x(\T \ulcorner P(\dot x)\urcorner \leftrightarrow P(x)) $
    \end{itemize}
\item The rule: 

\AxiomC{$A(P)$} \RightLabel{P-Subst; for $A$ in $\mc{L}(P)$ and $B$ in $\ltp$}
\UnaryInfC{$A(B/P)$}
\DisplayProof

\end{itemize}
\end{dfn}

In order to find a model for $\mrm{VFM}^*$, we generalize the construction of the minimal fixed point of $\mrm{VFM}$. Given some satisfaction relation $e$, we now write $X^P, Y^T \vDash_e \vphi$ as short for $(\nat, X, Y)\vDash_e \vphi$, where $X$ is the extension of $P$ and $Y$ the extension of $\T$. If the relation is classical, we omit any subscript. 

\begin{dfn}
$X^P, Y^T \vDash_{mc*} \vphi$ iff for all $Z^T\supseteq Y^T$ such that $Z^T\in \mrm{MAXCONS}$, $X^P, Z^T \vDash\vphi$.
\end{dfn}

We can then define: 

\begin{itemize}
    \item ${}_X\Gamma_0 =\varnothing$
    \item ${}_X\Gamma_{\alpha+1}= \{\vphi | X^P, {}_X\Gamma^T_\alpha\vDash_{mc*} \vphi\}$
    \item ${}_X\Gamma_\lambda =\bigcup_{\beta<\lambda} \Gamma_\beta$ for $\lambda$ a limit ordinal
\end{itemize}

\begin{prop}
There is some $\alpha\in\mrm{On}$ such that ${}_X\Gamma^T_\alpha={}_X\Gamma_{\alpha+1}$.
\end{prop}

The following result is easy to check:

\begin{prop}\label{P-Disq}
If $\alpha\in\mrm{On}$ is such that ${}_X\Gamma_\alpha={}_X\Gamma_{\alpha+1}$, then $X^P, {}_X\Gamma^T_\alpha \vDash$ P-Disq. 
\end{prop}

\begin{prop}\label{P-Subst}
For all $X$, $\mrm{VFM}^* \vdash \vphi \Rightarrow X^P, {}_X\Gamma^T_\alpha \vDash \vphi$
\end{prop}

\begin{proof}
The proof follows the line of \cite{feferman_1991}, proceeding by induction on the length of the proof in $\mrm{VFM}^*$. In light of Proposition \ref{prop:soundness VFM} and Proposition \ref{P-Disq}, what remains to be shown is the closure under P-Subst.  

Suppose $\mrm{VFM}^*\vdash A(P)$. By inductive hypothesis, one has $(\nat, X^P, {}_X\Gamma^T_\alpha)\vDash A(P)$ for all possible extensions $X$ for $P$. Since $A(P)$ is a formula of $\mc{L}(P)$, i.e., contains no instances of $\T$, the above is independent from the extension of $\T$; hence, we could consider any model of $\mc{L}(P)$ of the form $(\nat, Y)$ and obtain $(\nat, Y)\vDash A(P)$. But then one such $Y$ will be $Y=\{n \in \omega | (\nat, X^P, {}_X\Gamma^T_\alpha)\vDash B(\bar n) \}$; so indeed, it will follow that $(\nat, X^P, {}_X\Gamma^T_\alpha)\vDash A(B/P)$.
\end{proof}

\begin{corollary}
$\mrm{VFM}^*$ is consistent. 
\end{corollary}

In what follows, we provide the proof-theoretic analysis of $\mrm{VFM}^*$. 

\subsection{Lower bound}

We begin with a lower-bound for the proof-theoretic strength of $\mrm{VFM}^*$. This is just an observation of Fujimoto, namely \cite[Lemma 39]{fujimoto_2010}, together with the points we raised for the proof-theoretic lower-bound of $\mrm{VFM}$. We describe the proof, for convenience. 

\begin{lemma}\label{lower bound}
$| \mrm{VFM}^*| \geqq | \mrm{RT}_{<\Gamma_0} |\equiv |\mrm{RA}_{<\Gamma_0}|$
\end{lemma}

\begin{proof}
We begin by noting that the result we prove for the lower-bound of VFM, i.e., Theorem \ref{thm:lower_bound_VFM^-}, is in fact more general: what we showed, building on Fujimoto, is that any theory proving VF1 and VF5, together with the conditions on Lemma 1, can relatively truth-define as many ramified truth predicates as transfinite induction it can prove for $\mc{L}_\T$ (and this is Fujimoto's original result). So it will suffice to show that we can have transfinite induction for any ordinal $\alpha<\Gamma_0$ and any formula of $\ltp$.

Then, we can define the sequence $\{\beta_n| n<\omega\}$ as follows:

\begin{itemize}
    \item $\beta_0=\varepsilon_0$
    \item $\beta_{n+1}=\vphi_{\beta_n}(0)$
\end{itemize}

One can see that the supremum of this sequence is, precisely, $\Gamma_0$. So we will show, by induction, that $\mrm{VFM}^*\vdash \mrm{TI}_{\ltp}(<\beta_n)$ for all $n$. The case of $n=0$ is clear from the fact that we have full induction for $\ltp$. Assume then that $\mrm{VFM}^*\vdash \mrm{TI}_{\ltp}(<\beta_n)$, and we want to prove $\mrm{VFM}^*\vdash \mrm{TI}_{\ltp}(<\beta_{n+1})$. By the above, it follows that $\mrm{VFM}^*$ can relatively truth-define $\mrm{RT}_{<\beta_n}$. Now, thanks to Feferman's work (see, in particular, \cite{feferman_1968} and \cite{feferman_1991}), we know that $\mrm{RT}_{<\alpha}\vdash \mrm{TI}_{\lnat}(<\varphi_\alpha(0)) $. Hence, it follows that $\mrm{VFM}^*\vdash \mrm{TI}_{\mc{L}(P)}(<\vphi_{\beta_n}(0))$, that is, $\mrm{VFM}^*\vdash \mrm{TI}_{\mc{L}(P)}(<\beta_{n+1})$ (by definition). In particular, then, $\mrm{VFM}^*\vdash \mrm{TI}_{\mc{L}(P)}(<\beta_{n+1}, P)$. But then, since $\mrm{TI}_{\mc{L}(P)}(<\beta_{n+1}, P)$ is an arithmetical formula with parameter $P$, we can just apply the P-Subst rule to obtain $\mrm{VFM}^*\vdash \mrm{TI}(<\beta_{n+1}, B)$, for $B$ any formula of $\ltp$. Therefore, $\mrm{VFM}^*\vdash \mrm{TI}_{\ltp}(<\beta_{n+1})$. This completes the induction.

\end{proof}

\subsection{Upper bound}

Unsurprisingly, the techniques employed below will mimic those in Section \ref{section:upper_bound_vfm} and Section \ref{section:VFW}, and more particularly the latter. 

\begin{dfn}
$\patinf$ is the calculus comprising the basic axioms and rules of Definition \ref{dfn:vfminf}. $\patpinf$ is the calculus extending $\patinf$ with the following axiom:

\begin{equation*}
     \patpinf \sststile{k}{\alpha}\varGamma, \neg P(s), P(s\simeq t) \tag{Ax.3}
\end{equation*}
\end{dfn}

For a sequent $\varGamma:=\{\delta_0, ..., \delta_n\}$, let $\varGamma (B/P)$ be $\{\delta_0 (B/P), ..., \delta_n (B/P)\}$, i.e., the result of substituting $B$ for $P$ in every formula of $\varGamma$.

\begin{lemma}\label{admissibility of P-Subst in PA}
Let $A(X)$ be any formula of $\lp$, and $B$ any formula of $\ltp$. If $\patpinf \sststile{k}{\alpha} \varGamma, A(P)$, then $\patpinf \sststile{k}{\alpha\#\omega} \varGamma(B/P), A(B/P)$. Therefore, the P-Subst rule is admissible in $\patpinf$.
\end{lemma}

\begin{proof}
The proof is straightforward by induction on $\alpha$. One needs to consider that, for the base case, if $A(P)$ is active and follows from Axiom 1, then it must be of the form $s=t$ or $\neg s=t$, and so $A(P)=A(B/P)$. On the other hand, if $A(P)$ is active and follows from Axiom 3, then $A(P)$ is of the form $P(t)$ or $\neg P(t)$ for some term $t$, and so we might require up to $\omega$-more steps to prove $\varGamma(B/P), \neg B(t), B(s \simeq t)$.
\end{proof}

\begin{dfn}
$\vfmpinf$ is the calculus formulated in the language $\ltp$ and comprising all the rules and axioms of $\patpinf$ and $\mrm{VFM}^\infty$ plus the following two rules:
\end{dfn}

\begin{center}
\AxiomC{$\sststile{k}{\alpha_0}\varGamma, P(t^{\circ})$} \RightLabel{P-Disq$_1$}
\UnaryInfC{$\sststile{k}{\alpha}\varGamma, \T(s\simeq \subdot P t)$}
\DisplayProof    
\end{center}

\begin{center}
\AxiomC{$\sststile{k}{\alpha_0}\varGamma, \neg P(t^{\circ})$} \RightLabel{P-Disq$_2$}
\UnaryInfC{$\sststile{k}{\alpha}\varGamma, \T(s\simeq \subdot \neg\subdot P t)$}
\DisplayProof    
\end{center}

\begin{dfn}
$\mrm{VFM(P)}$ is the theory extending the theory $\mrm{VFM}$ with the axiom: 

\begin{equation*}\label{Disq-axiom}
    \forall t(\T(\subdot P t)\leftrightarrow Pt^\circ)\tag{P-Disq}
\end{equation*}
\end{dfn}

\begin{lemma}[Embedding]
If $\mrm{VFM(P)}\vdash \vphi$, then $\vfmpinf \sststile{0}{\alpha}\vphi$, for some $\alpha<\varepsilon_0$.
\end{lemma}
\begin{proof}
The proof follows the lines of the embedding proof for $\mrm{VFM}$. To recover \ref{Disq-axiom}, the only addenda, one first shows how the rule $\T_{\neq}$ together with the internal closure of $\T$ under logic yields the sequent $\varGamma, \neg\T(\subdot\neg\subdot P t), \neg\T (\subdot P t)$. Then, P-Disq$_1$, P-Disq$_2$, together with Ax.3, do the job. 
\end{proof}

\begin{dfn}
Let the extension of the predicate $P$ be $X$. The set $I^*_X \subseteq \{0,1\} \times \Gamma_0 \times \Gamma_0 \times \mrm{Seq}$ is defined to be the least fixed-point which is closed under the clauses in Definition \ref{dfn:I} plus the two additional clauses:

\begin{equation*}
I_X^*(i; \beta; \alpha; \varGamma, P(t)) \text{ if } t\in X\tag{Ax.4}
\end{equation*}
\begin{equation*}
I_X^*(i; \beta; \alpha; \varGamma, \neg P(t)) \text{ if } t\notin X \tag{Ax.5}
\end{equation*}

\noindent We write $I'(i;\alpha;\beta;\varGamma)$ instead of $\langle i,\alpha,\beta, \varGamma \rangle \in I'$.
\end{dfn}

We can now prove a few results about $I_X^*(x; y; z; w)$. They mimic the proofs for $I(x;y;z;w)$: 

\begin{lemma}
For any set $X\subseteq \omega$, the following holds: 

\begin{enumerate}
    \item (Substitution) If $I_X^*(i; \beta; \alpha; \varGamma, A(s))$ and $s=t$, then $I_X^*(i; \beta; \alpha; \varGamma, A(t))$. 
    \item (Weakening) Let $\alpha_0\leq \alpha, i\leq j, \beta_0\leq\beta, \varGamma_0\subseteq \varGamma$. If $I_X^*(i; \beta_0; \alpha_0; \varGamma_0)$, then $I_X^*(j; \beta; \alpha; \varGamma)$. 

\item (Cut-admissibility) If $I_X^*(i; \beta; \alpha_0; \varGamma, A)$ and $I_X^*(i; \beta; \alpha_1; \varDelta, \neg A)$, then $I_X^*(i; \beta; \omega_{\mathrm{co}(A)}(\alpha_0 \# \alpha_1); \varGamma, \varDelta)$, where $\mathrm{co}(A)$ is the logical complexity of $A$.
\end{enumerate}
\end{lemma}

\begin{proof}
All proofs proceed by induction on the second ordinal index, $\alpha$. In the case of 8, the successor case requires a side-induction on the complexity of the formula $A$. We just mention some steps in that proof that might not be obvious: 

\textit{When proving the claim for $\alpha=0$}, we deal with the case of (Comp). We will then have $I_X^*(i; \beta; 0; \varGamma, \T\ulcorner A\urcorner)$ (the case for $\varGamma, \T\ulcorner \neg A\urcorner$ will be symmetric). We also have $I_X^*(i; \beta; 0; \varDelta, \neg\T\ulcorner A\urcorner)$. We focus on the case in which $\neg\T\ulcorner A\urcorner$ is the active formula, since otherwise it is straightforward. Since the second ordinal index is 0, $I_X^*(i; \beta; 0; \varDelta, \neg\T\ulcorner A\urcorner)$ must follow from (Ax.3), so $\T\ulcorner A\urcorner\in \varDelta$. Thus, weakening on $I_X^*(i; \beta; 0; \varGamma, \T\ulcorner A\urcorner)$, we obtain $I_X^*(i; \beta; 0; \varGamma, \varDelta)$.

\textit{When proving the claim for $\alpha=\gamma+1$, $\vert A\vert>0$ and $A$ the active formula}, one needs to examine two cases. We sketch how it works for the case of $A=A_0\wedge A_1$. We have $I_X^*(i; \beta; \alpha'_0; \varGamma, A_0)$, as well as $I_X^*(i; \beta; \alpha'_0; \varGamma, A_1)$, $\alpha'_0<\alpha_0$ and $I_X^*(i; \beta; \alpha_1; \varDelta, \neg A_0\vee\neg A_1)$. Then the two cases to distinguish are: (i) the case in which $I_X^*(i; \beta; \alpha_1; \varDelta, \neg A_0\vee\neg A_1)$ is obtained by $(\vee)$; and (ii) the case in which it is obtained by (Cons) or (Norm). For (i), the claim will follow by IH (twice) and weakening. For (ii), one uses (5) above on the premise(s) of (Cons)/(Norm), employs again IH twice and weakening, and applies (Cons)/(Norm). 
\end{proof}

\begin{dfn}
    A sequent $\varGamma$ is atomic if it consists only of atomic sentences of $\ltp$, i.e., equations, or formulae of the forms $P(t)$ or $\T(t)$. 
\end{dfn}

\begin{dfn}
    For an atomic sequent $\varGamma$, we define its disquotation, $\mrm{Dis}(\varGamma)$, as the sequent $\{s=t\vert s = t\in \varGamma\} \cup \{P(t)\vert P(t)\in \varGamma\}\cup \{A \in \ltp \vert t \text{ is the code of  } A \text{ and } \subdot \T(t)\in \varGamma\}$.
\end{dfn}

\begin{lemma}[Disquotation]
For any set $X\subseteq\omega$ and for any atomic sequent $\varGamma$, assume $I_X^*(i; \beta; \alpha; \varGamma)$. Then $I_X^*(i; \beta_0; \omega_n(\alpha); \mrm{Disq}(\varGamma))$ holds for some $n\in\nat$ and $\beta_0<\beta$, provided $\beta\geq 1$. Moreover, if $\varGamma$ contains only equations, $I_X^*(i; 0; \varepsilon_\alpha; \varGamma)$ holds. 
\end{lemma}
\begin{proof}
The proof, of course, works exactly like the one for $I(x;y;z;w)$. It proceeds then by induction on $\alpha$. We treat the case of (Ax.4) as the case of (Ax.1); the case of (Ax.5), just like (Ax.3), cannot arise, as otherwise $\varGamma$ is not atomic. The `moreover' part of the claim follows directly from the first part. 
\end{proof}

\begin{lemma}[Elimination of (Cons) and (Norm)]
For any set $X\subseteq\omega$, if $I_X^*(i; 0; \alpha; \varGamma)$ holds and $\varGamma$ is atomic, then $I_X^*(0; 0; \omega_n(\alpha); \mrm{Dis}(\varGamma))$ holds for some $n\in\nat$. Therefore, if $\varGamma$ contains only equations, we obtain $I_X^*(0; 0; \omega_n(\alpha); \varGamma)$.
\end{lemma}
\begin{proof}
By inducting on $\alpha$. Nothing relevant changes w.r.t. the same proof for $I(x;y;z;w)$.
\end{proof}
\begin{lemma}[Consistency]
For any set $X\subseteq\omega$, if $I_X^*(i; \beta;\alpha;s=t)$ holds, then $s=t$. Likewise, if $I_X^*(i; \beta;\alpha;P(t))$ holds, then $t\in X$. 
\end{lemma}

Finally, we define the interpretation that yields the soundness of $\vfmpinf$:

\begin{dfn}
For a set $X\subseteq\omega$, the relation $\vDash^\beta_{X}A$ is inductively defined by the clauses for $\vDash^\beta A$  for atomic and negated atomic sentences of $\lnat$, $\wedge$, $\vee$, $\forall$ and $\exists$ plus the following clauses:

\begin{itemize}
    \item $\vDash^\beta_X P(t):\Leftrightarrow t\in X$
    \item $\vDash^\beta_X \neg P(t) :\Leftrightarrow t\notin X$
    \item $\vDash^\beta_X \T(t):\Leftrightarrow$ the value of $t$ is the code of a sentence $A$ and $I_X^*( 1; \beta; \varepsilon_\beta; A)$ holds. 
\end{itemize}

\end{dfn}

\begin{thm}[Soundness]
Let $\varGamma$ be a $\T$-positive sequent. If $\vfmpinf\sststile{0}{\alpha}\varGamma$, then $\vDash^\alpha_X \bigvee \varGamma$, for all $X\subseteq \omega$. 
\end{thm}

\begin{proof}
The proof proceeds by induction on $\alpha$, and is just like the soundness theorem for $\mrm{VFM}^\infty$. 

For (Ax.3): clearly, for any $t$, either $t \in X$ or $t \notin X$. If the former, then $\vDash^\beta_X P(t)$ holds for all $\beta\in \mrm{On}$ (including 0), so $\vDash^\beta_X \bigvee \varGamma\vee\neg P(s)\vee P(s\simeq t)$ holds for all $\beta$. 

For the new rules, we can reason as follows. Take the case of P-Disq$_1$. By IH, $\vDash^{\alpha_0}_X \bigvee \varGamma \vee P(t)$ holds for any $X\subseteq\omega$. We can assume that $\vDash^{\alpha_0}_X P(t)$, as otherwise we are done. So, by the appropriate clause, $t \in X$, whence $I_X^*(i; \gamma_0; \gamma_1; P(t))$ follows for any $\gamma_0, i, \gamma_1$ by (Ax.4). Instantiating, $I_X^*(1; \alpha; \varepsilon_\alpha; P(t))$ obtains, and hence so does $\vDash^\alpha_X \T(\subdot P(t))$. An identic reasoning would give us P-Disq$_2$.
\end{proof}

\begin{lemma}
Let $A$ be a $\T$-positive formula of $\ltp$. Let
$\mrm{VFM(P)}\vdash A$. Then, for all $X\subseteq \omega$, $\vDash^\alpha_X A$ holds for some $\alpha<\varepsilon_0$.
\end{lemma}

In what follows, we work with ramified theories of truth over the language $\lp$, that is: for an ordinal $\gamma\leq\Gamma_0$, $\mc{L}(P)_{<\gamma}$ is defined as $\mc{L}(P)_{<\gamma}=\mc{L}_{<\gamma}\cup\{P\}$, or else $\lp$ if $\gamma=0$. $\mc{L}(P)_{\gamma}$ is defined as $\mc{L}(P)_{\gamma}=\mc{L}_{\gamma}\cup\{P\}$.

\begin{dfn}
For an ordinal $\gamma\leq\Gamma_0$, we write $\mrm{RT}_{<\alpha} \llbracket \mrm{PA(P)}\rrbracket$ for the theory in the language  $\mc{L}(P)_{<\alpha}$ given by the axioms of $\mrm{RT}_{<\alpha}$ as defined in e.g. \cite[Def.9.2]{halbach_2014} plus the axiom: 

\begin{equation*}
\forall t (\T_\beta(\subdot P t) \leftrightarrow Pt^\circ), \text{ for all }\beta<\alpha \tag{RTP}
\end{equation*}
\end{dfn}

\begin{lemma}\label{rt to pa}
For any formula $A(P)\in \lp$, 

\begin{center}
$\mrm{RT}_{<\alpha} \llbracket \mrm{PA(P)}\rrbracket \vdash A(P) \Rightarrow \patpinf\sststile{0}{\vphi_{\gamma+1}0} A(P)$,
\end{center}

\noindent where $\gamma$ is the lowest ordinal such that $\gamma=\omega^\gamma$ and $\alpha\leq\gamma$.

\end{lemma}

\begin{proof}
The proof in \cite[\S 3]{hayashi_2022} can easily be adapted to our context, so we obtain: 

\begin{center}
$\mrm{RT}_{<\omega^\alpha} \llbracket \mrm{PA(P)}\rrbracket \vdash A(P) \Rightarrow \patpinf\sststile{0}{\vphi_{\alpha+1}0} A(P)$   
\end{center}

\noindent Then one just needs to consider the different cases that may occur: 

\begin{itemize}
    \item If $\alpha<\varepsilon_0$, then clearly 
    \begin{align*}
    \mrm{RT}_{<\alpha} \llbracket \mrm{PA(P)}\rrbracket \vdash A(P)&\Rightarrow \mrm{RT}_{<\omega^\alpha} \llbracket \mrm{PA(P)}\rrbracket \vdash A(P) \\
    & \Rightarrow \patpinf\sststile{0}{\vphi_{\alpha+1}0} A(P)\\
    & \Rightarrow \patpinf\sststile{0}{\vphi_{\varepsilon_0+1}0} A(P)
    \end{align*}
    
    \item If $\alpha\geq \varepsilon_0$ and $\alpha=\omega^{\alpha}$, then 
    
    \begin{align*}
    \mrm{RT}_{<\alpha} \llbracket \mrm{PA(P)}\rrbracket \vdash A(P)&\Leftrightarrow \mrm{RT}_{<\omega^\alpha} \llbracket \mrm{PA(P)}\rrbracket \vdash A(P) \\
    & \Rightarrow \patpinf\sststile{0}{\vphi_{\alpha+1}0} A(P)\\
    \end{align*}
    
    \item If $\alpha \geq \varepsilon_0$ but $\alpha\neq \omega^\alpha$, there are $\eta_0, \eta_1$ such that $\omega^{\eta_0}=\eta_0<\alpha<\eta_1=\omega^{\eta_1}$, and so 

    \begin{align*}
    \mrm{RT}_{<\alpha} \llbracket \mrm{PA(P)}\rrbracket \vdash A(P)&\Rightarrow \mrm{RT}_{<\eta_1} \llbracket \mrm{PA(P)}\rrbracket \vdash A(P)\\
    & \Leftrightarrow \mrm{RT}_{<\omega^{\eta_1}} \llbracket \mrm{PA(P)}\rrbracket \vdash A(P) \\
    & \Rightarrow \patpinf\sststile{0}{\vphi_{{\eta_1}+1}0} A(P)\\
    \end{align*}

\end{itemize}

\noindent One can see that, in all cases, we just take the first $\gamma$ meeting $\gamma\geq \alpha$ and $\omega^\gamma=\gamma$. 
\end{proof}

Let $\hat{\alpha}$ be the least fixed point of $\varepsilon_x$ larger than $\alpha$, i.e., we write $\hat{\alpha}$ as short for $\mrm{min}\{\beta \in \mrm{On} \sth \beta> \alpha \text{ and } \varepsilon_\beta=\beta\}$.  In particular, $\hat{0} = \varphi_20$. By inspecting the proof of Theorem \ref{thm:upper_VFW}, we realize that:

\begin{lemma}\label{formalizing in rt}
If $\vDash^\alpha_X A$, for $A\in\lp$, then $\mrm{RT}_{<\hat{\alpha}}\llbracket \mrm{PA(P)}\rrbracket\vdash A$. 
\end{lemma}

\begin{lemma}\label{formalization in RT}
Let $A(P)\in\lp$. If $\mrm{VFM(P)}\vdash A(P)$, then $\patpinf\sststile{0}{\vphi_{\vphi_2 0+1}0} A(P)$.
\end{lemma}

\begin{proof}
If $\mrm{VFM(P)}\vdash A(P)$, then $\vfmpinf\sststile{0}{\alpha}A(P)$ with $\alpha<\varepsilon_0$ by the embedding. So, for any $X$, $\vDash_X^\alpha A(P)$ by soundness. Then, by Lemma \ref{formalizing in rt}, we get $\mrm{RT}_{<\vphi_20}\llbracket \mrm{PA(P)}\rrbracket\vdash A(P)$.  Finally, the claim follows by Lemma \ref{rt to pa}, since $\omega^{\vphi_20}=\vphi_20$.
\end{proof}

Now, for the next result, note that, clearly, $\omega^{\hat{\alpha}}=\hat{\alpha}$---see e.g. \cite[Lemma 3.4.8] {pohlers_2009}. 

\begin{lemma}\label{from vfmpinf to painf}
Let $A(P)\in\lp$. If $\vfmpinf\sststile{0}{\alpha}A(P)$, then $\patpinf\sststile{0}{\varphi_{\hat{\alpha}+1}0\#\omega} A(B/P)$; therefore, $\vfmpinf\sststile{0}{\varphi_{\hat{\alpha}+1}0\#\omega}A(B/P)$. Hence, the P-Subst rule is admissible in $\vfmpinf$.
\end{lemma}
\begin{proof}
The `therefore' claim follows from the fact that $\patpinf$ is a subtheory of $\vfmpinf$. For the first claim, we distinguish two cases. 

If $\alpha<\vphi_20$, then the assumption yields $\vDash^\alpha_X A(P)$, and so by Lemma \ref{formalizing in rt}, $\mrm{RT}_{<\vphi_20}\llbracket \mrm{PA(P)}\rrbracket\vdash A(P)$, whence $\patpinf\sststile{0}{\vphi_{\vphi_20+1}0} A(P)$ by Lemma \ref{rt to pa}. Then, apply Lemma \ref{admissibility of P-Subst in PA}. 

If $\alpha\geq\vphi_20$, the assumption yields $\vDash^\alpha_X A(P)$. The upper bound on the height of the derivation in $I^*_X(x;y;z;w)$ that needs to be formalized via ramified predicates is $\varepsilon_{\alpha}$. Taking $\hat{\alpha}$, $\varepsilon_{\hat{\alpha}}=\hat{\alpha}$, and $\varepsilon_\alpha < \varepsilon_{\hat{\alpha}}$. Hence, $\mrm{RT}_{<\hat{\alpha}}\llbracket \mrm{PA(P)}\rrbracket$ suffices to derive $A(P)$. Then, $\patpinf\sststile{0}{\vphi_{\hat{\alpha}+1}0} A(P)$ by Lemma \ref{rt to pa}. Finally, apply Lemma \ref{admissibility of P-Subst in PA}.
\end{proof}


\begin{corollary}\label{vfminf to pa}
Let $A\in \lnat$. If $\vfmpinf\sststile{0}{\alpha}A$, then $\patpinf\sststile{0}{\vphi_{\hat{\alpha}+1}0} A$.
\end{corollary}
\begin{proof}
If $\vfmpinf\sststile{0}{\alpha}A$, then $\vDash^\alpha_X A$. Then, $\mrm{RT}_{< \hat{\alpha}} \llbracket \mrm{PA(P)}\rrbracket \vdash A$. Finally, we apply Lemma \ref{rt to pa}.
\end{proof}

\begin{thm}[Upper-bound for $\mrm{VFM}^*$]
$\mrm{VFM}^*\leq \mrm{PA}+\mrm{TI}(<\Gamma_0)$. 
\end{thm}

\begin{proof}
 Define first the following sequence, for all $n\in\omega$: 

    \begin{itemize}
        \item $f(0)=\varphi_{\vphi_20+1}0$
        \item $f(n)=\varphi_{f(n-1)}0$
    \end{itemize}
    
Now, we claim that, if $A \in \lnat$ is derivable in $\mrm{VFM}^*$ with $n$-many applications of P-Subst, $\patpinf\sststile{0}{f(n+2)} A$. For that, we need to induct on $n$. 

If $n=0$, the proof has in fact been carried out in $\mrm{VFM(P)}$, so Lemma \ref{formalization in RT} yields the result. 

If $n=m$, we assume the claim for $(m-1)$. For the last application of the P-Subst, the IH yields that $\patpinf\sststile{0}{f(m+1)} A(P)$ holds. So $\vfmpinf\sststile{0}{f(m+1)} A(P)$ also holds. By Lemma \ref{from vfmpinf to painf}, $\vfmpinf\sststile{0}{f(m+1)} A(P)$ implies $\patpinf\sststile{0}{\varphi_{\hat{\alpha}+1}0\#\omega} A(B/P)$, for $B\in \lnat$. Now note that, for any $\alpha=f(n), n>0$, $\varphi_{\hat{\alpha}+1}0\#\omega<f(n+1)$. So one can conclude $\patpinf\sststile{0}{f(m+2)} A(B/P)$. This completes the claim. 

Now, it is easy to see that the limit of the sequence $f(n)$ is $\Gamma_0$. Since any proof in $\mrm{VFM}^*$ contains at most $n$-many applications of P-Subst, if $A \in \lnat$ is derivable in $\mrm{VFM}^*$, then $\patpinf\sststile{0}{\alpha} A$ for some $\alpha<\Gamma_0$. Finally, derivations of $\lnat$ formulae in $\patpinf$ of length up to $\Gamma_0$ can be formalized in $\mrm{PA}+\mrm{TI}(<\Gamma_0)$. Therefore, the theorem follows. 
\end{proof}

By the results obtained in this section: 

\begin{thm}
$\vert\mrm{VFM}^*\vert \equiv \vert\mrm{PA}+\mrm{TI}(<\Gamma_0)\vert \equiv \vert \mrm{KF}^*\vert $
\end{thm}

\section{Concluding remarks}

In this paper, we aimed to present axiomatic counterparts for the semantic supervaluational theories of truth in the style of Kripke which had not been addressed in the literature: the theories VB and MC. This aim has only been partially accomplished, because indeed it can only be partially accomplished. By results obtained by Fischer et al. \cite{fischer_halbach_kriener_stern_2015}, we know that no axiomatic theory is actually a good axiomatization of VB, nor of MC, nor of any semantic supervaluational theory in the literature; at least not if by `good axiomatization' we understand an axiomatic theory $S$ such that $S$ is $\nat$-categorical with respect to the supervaluational operator $\Phi$, i.e., if 

\begin{equation*}
    (\nat, X)\vDash S \Leftrightarrow X=\Phi(X).
\end{equation*}

Given this constraint,\footnote{In particular, the exact result is that for any theory based on a scheme $\Phi$ such that $\mrm{SV}(X)\subseteq \Phi(S)\subseteq \mrm{MC}(X)$, $\nat$-categoricity fails.} our axiomatic theories as good as one can expect for supervaluational theories, insofar as both $\mrm{VF}^-$ and $\mrm{VFM}$ are sound with respect to, respectively, VB and MC; and they are not trivially so.\footnote{For instance, $\pat$ is also sound with respect to all semantic supervaluational theories, but it is trivially so.}

Besides the above, in the paper we have introduced variations of those theories and, in all cases, we have provided a proof-theoretic analysis. We can sum up the results we obtained with the following table:

\begin{table}[H]
\centering
\begin{tabular}{|p{2cm}|p{2cm}|p{3cm}|p{2.5cm}|}
\hline
Proof-theoretic ordinal & First-order arithmetical theory & Supervaluational system & Compositional system  \\ \hline
   $\psi_0(\varepsilon_{\Omega+1})$                 &   $\mrm{ID}_1$ & VF / $\mrm{VF}^-$ / $\mrm{VF}^*$ $\mrm{VF}^{-*}$                    &       ??        \\ \hline
         $\Gamma_0$               &     $\widehat{\mrm{ID}}_{<\omega}$         & $\mrm{VFM}^*$           &     $\mrm{KF}^*$          \\ \hline
         $\vphi_{\vphi_20}0$               &     ??        & $\mrm{VFW}$           &     ??          \\ \hline
    $\vphi_{\varepsilon_0}0$               &     $\widehat{\mrm{ID}}_1$         & $\mrm{VFM}$           &     $\mrm{KF}$          \\ \hline
          
\end{tabular}
\end{table}

Arguably, out of the four supervaluational schemes (SV, VB, VC and MC), VC and MC enjoy a privileged position: they require, and uniformly so, that truth be either consistent (VC) or consistent and complete (MC). Of course, this is reflected on the axiomatic theories, as these requirements are imposed on the internal theories of VF and VFM. Now, while perhaps meeting both requirements at once would be desirable, our results cast doubt on the possibility to adjuciate so quickly in favour of VFM. For if one believes that proof-theoretic strength is one of the desiderata for axiomatic theories of truth, our proof-theoretic analysis for VFM points in the other direction, giving us reasons to reject $\mrm{VFM}$ in favour of $\mrm{VF}$ or even the weaker $\mrm{VFW}$. 

In fact, the proof-theoretic analyses of VF, VFW and VFM do seem to suggest that internal completeness severely limits the proof-theoretic strength. Admittedly, the latter can be lifted by considering the schematic extension, but the resulting theory still falls short of the strength of VF. 

What's more: by examining the theories presented by Friedman and Sheard in \cite{friedman_sheard_1987} and their respective analyses in \cite{leigh_rathjen_2012}, as well as considering other axiomatic theories in the literature that include the completeness (T-Comp) axiom, such as KF+(T-Comp), we notice that none of them surpasses the proof-theoretic limits of predicative analysis---or of $\widehat{\mrm{ID}}_1$, for that matter. So an immediate open question arises: 

\begin{op}
    Is there a ``natural'' axiomatic theory of truth that includes the axiom of $\mrm{(T}$-$\mrm{Comp)}$, either in the internal or the external theory, with proof-theoretic ordinal $>\vphi_{\varepsilon_0}(0)$? And with proof-theoretic ordinal $>\Gamma_0$?
\end{op}

While it is obvious that the meaning of `natural' is difficult to spell out, it is also clear to us that schematic extensions fall outside of this category. 

A related but distinct question concerns supervaluational theories in general. Thus, having seen how the proof-theoretic strength of VFM decreases substantially with respect to VF and $\mrm{VF}^-$, we wonder whether a supervaluational scheme can be produced so that it is more restrictive than VB \textit{and} the corresponding axiomatic theory is stronger than VFM. In technical terms: 

\begin{op}
Is there an admissibility condition $\Phi(\cdot)$ such that:

\begin{itemize}
    \item $\{\#\vphi \sth \forall Y\supseteq X (\Phi(Y)\Ra Y\vDash \vphi)\}\subsetneq \mrm{VB}(X)$ for $X$ a consistent set of sentences,
    \item there is some axiomatic theory of truth $\Sigma$ such that $(\nat, X)\vDash \Sigma$ when $X=\{ \vphi \sth \forall Y\supseteq X(\Phi(Y) \Ra Y\vDash \vphi)\}$, and
    \item $\vert \Sigma\vert > \vphi_{\varepsilon_0}0$?
\end{itemize}
\end{op}

\noindent Finally, we list a couple of further technical questions that arise out of our project: 

\begin{op}
How expressively strong is the truth predicate of $\mrm{VFM}$?
For example, can $\mrm{VFM}$ define the truth predicate of $\mrm{KF}$ in the sense of \cite{fujimoto_2010}?
\end{op}

\begin{op}
Can the upper-bound of $\mrm{VFM}$ be obtained via a direct interpretation of the theory in some system of first- or second-order arithmetic, or a known theory of truth?
\end{op}

\begin{op}
Can the upper-bound proof of $\mrm{VFM}$, for which we offer a formalization in $\mrm{ID}^*_1$, be formalized in $\mrm{RT}_{<\varepsilon_0}$?
\end{op}


\footnotesize \subsection*{\small Acknowledgements} The authors would like thank audiences at the 1st KCL/SNS Pisa Logic and Phil of Maths meeting, the 2024 Logic Colloquium and the XIII Workshop in Philosophical Logic at CONICET/University of Buenos Aires. We would particularly like to thank Carlo Nicolai for helpful comments. The work of the first author was made possible by an LAHP (London Arts and Humanities Partnership) studentship, as well as by PLEXUS (Grant Agreement no 101086295), a Marie Sklodowska-Curie action funded by the EU under the Horizon Europe Research and Innovation Programme.

\normalsize

\bibliography{bibliography}
\bibliographystyle{plain}
\end{document}